\documentclass[11pt,a4paper]{article}

\usepackage{mathrsfs}
\usepackage{amsfonts}
\usepackage{}

\usepackage[leqno]{amsmath}
\usepackage{graphicx}
\usepackage{latexsym}
\usepackage{amsmath,amsfonts,amssymb,amsthm,mathrsfs,euscript,makeidx,color}
\usepackage{enumerate}

\usepackage[colorlinks,linkcolor=blue,anchorcolor=green,citecolor=red]{hyperref}

\def\sqr#1#2{{\vcenter{\vbox{\hrule height.#2pt
              \hbox{\vrule width.#2pt height#1pt \kern#1pt \vrule width.#2pt}
              \hrule height.#2pt}}}}
\def\5n{\negthinspace \negthinspace \negthinspace \negthinspace \negthinspace }
\def\4n{\negthinspace \negthinspace \negthinspace \negthinspace }
\def\3n{\negthinspace \negthinspace \negthinspace }
\def\2n{\negthinspace \negthinspace }
\def\1n{\negthinspace }

\def\dbE{\mathbb{E}}
\def\dbF{\mathbb{F}}

\def\dbH{\mathbb{H}}

\def\dbL{\mathbb{L}}

\def\dbP{\mathbb{P}}

\def\dbR{\mathbb{R}}

\def\dbT{\mathbb{T}}

\def\dbX{\mathbb{X}}

\def\cA{{\cal A}}

\def\cF{{\cal F}}

\def\cL{{\cal L}}

\def\cT{{\cal T}}
\def\cU{{\cal U}}

\def\cX{{\cal X}}
\def\cY{{\cal Y}}
\def\cZ{{\cal Z}}

\def\by{{\bf y}}
\def\ch{\textsc{h}}

\def\as{\hbox{\rm a.s.}}

\def\ds{\displaystyle}

\def\ns{\noalign{\ss}}
\def\no{\noindent}

\def\ss{\smallskip}
\def\ms{\medskip}

\def\q{\quad}
\def\qq{\qquad}
\def\hb{\hbox}

\def\({\Big (}
\def\){\Big )}
\def\[{\Big[}
\def\]{\Big]}

\def\lan{\langle}
\def\ran{\rangle}

\def\rf{\eqref}

\def\a{\alpha}
\def\b{\beta}

\def\d{\delta}
\def\e{\varepsilon}

\def\k{\kappa}
\def\l{\lambda}

\def\si{\sigma}
\def\t{\tau}
\def\f{\varphi}
\def\th{\theta}
\def\bx{\O}

\def\i{\infty}
\def\G{\Gamma}
\def\D{\Delta}

\def\L{\Lambda}

\def\F{\Phi}
\def\bx{\O}

\def\da{\mathop{\downarrow}}

\def\h{\widehat}
\def\wt{\widetilde}

\def\cd{\cdot}

\def\cds{\cdots}

\def\tr{\hbox{\rm tr$\,$}}

\def\les{\leqslant}
\def\ges{\geqslant}
\def\pa{\partial}
\def\bde{\begin{definition}\label}
\def\ede{\end{definition}}
\def\be{\begin{equation}}
\def\bel{\begin{equation}\label}
\def\ee{\end{equation}}
\def\bt{\begin{theorem}\label}
\def\et{\end{theorem}}
\def\bc{\begin{corollary}\label}
\def\ec{\end{corollary}}
\def\bl{\begin{lemma}\label}
\def\el{\end{lemma}}
\def\bp{\begin{proposition}\label}
\def\ep{\end{proposition}}
\def\bas{\begin{assumption}\label}
\def\eas{\end{assumption}}
\def\br{\begin{remark}\label}
\def\er{\end{remark}}
\def\bex{\begin{example}\label}
\def\ex{\end{example}}
\def\ba{\begin{array}}
\def\ea{\end{array}}
\def\ben{\begin{enumerate}}
\def\een{\end{enumerate}}

\def\square#1{\vbox{\hrule\hbox{\vrule height#1%
     \kern#1\vrule}\hrule}}
\def\rectangle#1#2{\vbox{\hrule\hbox{\vrule height#1%
     \kern#2\vrule}\hrule}}

\font\tenbb=msbm10 \font\sevenbb=msbm7 \font\fivebb=msbm5

\newfam\bbfam
\scriptscriptfont\bbfam=\fivebb \textfont\bbfam=\tenbb
\scriptfont\bbfam=\sevenbb

\newcommand{\bea}{\begin{eqnarray}}
\newcommand{\eea}{\end{eqnarray}}
\newcommand{\beaa}{\begin{eqnarray*}}
\newcommand{\eeaa}{\end{eqnarray*}}
\def\bx{{\bf x}}
\def\b0{{\bf 0}}

\newtheorem{theorem}{\indent Theorem}[section]
\newtheorem{definition}[theorem]{\indent Definition}
\newtheorem{proposition}[theorem]{\indent Proposition}
\newtheorem{corollary}[theorem]{\indent Corollary}
\newtheorem{lemma}[theorem]{\indent Lemma}
\newtheorem{remark}[theorem]{\indent Remark}
\newtheorem{example}[theorem]{\indent Example}

\newtheorem{assumption}[theorem]{\indent Assumption}

\makeatletter
   
   \@addtoreset{equation}{section}
\makeatother

\makeatletter 

\@addtoreset{equation}{section}
\makeatother 

\begin{document}
\title{\bf Path Dependent Feynman--Kac Formula for Forward Backward Stochastic Volterra Integral Equations}

\author{
Hanxiao Wang
\footnote{Department of Mathematics, National University of Singapore,
                           Singapore 119076, Singapore (Email: \texttt{ mathxw@nus.edu.sg}).
                           This author is supported by Singapore MOE AcRF Grants R-146-000-271-112.},~~
Jiongmin Yong
\footnote{Department of Mathematics, University of Central Florida, Orlando, FL 32816 USA (Email: \texttt{jiongmin.yong@ucf.edu}). This author is supported in part by NSF Grant  DMS-1812921. }~~
and~~Jianfeng Zhang
\footnote{Department of Mathematics, University of Southern California, Los Angeles,  CA 90089 USA (E-mail: \texttt{jianfenz@usc.edu}). This author is supported in part by NSF Grant DMS-1908665. }}

\maketitle

\begin{quote}

\footnotesize {\bf Abstract.}
This paper is concerned with the relationship between forward-backward stochastic Volterra integral equations (FBSVIEs, for short) and a system of (non-local in time) path dependent partial differential equations (PPDEs, for short). Due to the nature of Volterra type equations, the usual flow property (or semigroup property) does not hold.
Inspired by Viens--Zhang \cite{Viens-Zhang-2019} and Wang--Yong \cite{Wang-Yong-2019}, auxiliary processes are introduced so that the flow property of adapted solutions to the FBSVIEs is recovered in a suitable sense, and thus the functional It\^o formula is applicable. Having achieved this stage, a natural PPDE is found so that the adapted solution of the backward SVIEs admits a representation in terms of the solution to the forward SVIE via the solution to a PPDE.
On the other hand, the solution of the PPDE admits a representation in terms of adapted solution to the (path dependent) FBSVIE, which is referred to as a Feynman--Kac formula. This leads to the existence and uniqueness of a classical solution to the PPDE, under smoothness conditions on the coefficients of the FBSVIEs. Further, when the smoothness conditions are relaxed with the backward component of FBSVIE being one-dimensional, a new (and suitable) notion of viscosity solution is introduced for the PPDE, for which a comparison principle of the viscosity solutions is established, leading to the uniqueness of the viscosity solution. Finally, some results have been extended to coupled FBSVIEs and type-II BSVIEs, and a representation formula for the path derivatives of  PPDE solution is obtained by a closer investigation of linear FBSVIEs.
\ms

{\bf Keywords.} Forward-backward stochastic Volterra integral equation,
path dependent partial differential equation, Feynman--Kac formula,
viscosity solution, comparison principle.

\ms

{\bf AMS Subject Classifications.}  60H20,  45D05, 35K10, 35D40,  60G22.
\end{quote}


\section{Introduction}
\label{sect-Introduction}
\setcounter{equation}{0}
Let $(\Omega,\cF,\dbF,\dbP)$ be a complete filtered probability space, $W$ a  $d$-dimensional standard Brownian motion, $\dbF\equiv\{\cF_t\}_{t\ges0}$ the natural filtration generated by $W$ augmented by all the $\dbP$-null sets in $\cF$, and $T>0$ a fixed time horizon. Consider the following (decoupled) {\it forward-backward stochastic differential equation} (FBSDE, for short): given initial data $(t,x) \in [0, T]\times \dbR^n$,
\bel{FBSDE1}\left\{\ba{ll}
\ds X^{t,x}_s= x + \int_t^s b(r,X^{t,x}_r)dr+\int_t^s \si(r,X^{t,x}_r)dW_r,\\
\ds Y^{t,x}_s= g(X^{t,x}_T) +\int_s^T f(r,X^{t,x}_r,Y^{t,x}_r,Z^{t,x}_r)dr - \int_t^T Z^{t,x}_rdW_r,\ea\right. s\in[t,T],\ee
where the coefficients $b,\si,f,g$ are deterministic functions. Such an FBSDE is associated with the following terminal value problem  of a {\it partial differential equation} (PDE, for short):
\bel{PDE1}\left\{\ba{ll}
\ds\pa_t u(t,x)+{1\over2}\tr\big[\pa_{xx}^2u(t,x)\si(t,x)\si^\top(t,x)\big]+\pa_xu(t,x) b(t,x)\\
\ns\ds\qq+f(t,x,u(t,x),\pa_xu(t,x)\si(t,x)\big)=0,\qq(t,x)\in[0,T]\times\dbR^n,\\
\ns\ds u(T,x)=g(x),\qq x\in\dbR^n. \ea\right.\ee
By the seminal works Peng \cite{Peng-1991} and Pardoux--Peng \cite{Pardoux-Peng-1992},  we have the  nonlinear {\it Feynman--Kac formula}, representing the viscosity solution to PDE (\ref{PDE1}) by the adapted solution to FBSDE (\ref{FBSDE1}):
\bel{U=Y}
u(t,x)=Y^{t,x}_t,\qq(t,x)\in[0,T]\times\dbR^n,
\ee
and on the other hand the adapted solution $(Y^{t,x},Z^{t,x})$ to the {\it backward stochastic differential equation} (BSDE, for short), namely the second equation in \rf{FBSDE1}, has the following {\it representation formula} via the solution to  PDE \rf{PDE1}:
\bel{Y=U}
Y^{t,x}_s=u(s,X_s^{t,x}),\qq Z^{t,x}_s=\pa_xu(s,X^{t,x}_s)\si(s,X^{t,x}_s),\qq s\in[t,T],
\ee
provided  $u$ is smooth.  The key for this PDE approach is the {\it flow property}, also called {\it semigroup property} and can be viewed as a type of {\it time consistency}, of the FBSDE. That is,
\bel{semigroup1}
X^{t,x}_r = X^{s, X^{t,x}_s}_r,\q
Y^{t,x}_s = Y^{r, u(r,\cd); t,x}_s,\q  Z^{t,x}_s = Z^{r, u(r,\cd); t,x}_s,\qq t\les s \les r\les T,
\ee
where $(Y^{r, u(r,\cd); t,x}, Z^{r, u(r,\cd); t,x})$ is the solution to the BSDE on $[t, r]$ with terminal condition $Y_r = u(r, X^{t,x}_r)$. We  remark that this approach remains effective for coupled FBSDEs (namely $b, \si$ may depend on $(Y, Z)$), see Ma--Protter--Yong \cite{Ma-Protter-Yong-1994}, and even for more general situations, where $u$ plays the role of the {\it decoupling field} for the forward-backward equations.

In this paper, our objective is to consider the following decoupled {\it forward-backward stochastic Volterra integral equation} (FBSVIE, for short) with solution triple $(X_t, Y_t, Z^t_r)$, $0\les t\les r\les T$:
\bel{FBSVIE1}
\left.\ba{ll}
\ds X_t=\bx_t+\int_0^tb(t,r,X_r)dr+\int_0^t\si(t,r,X_r)dW_r,\\
\ds Y_t=g(t,X_T)+\int_t^Tf(t,r,X_r,Y_r,Z_r^t)dr-\int_t^TZ_r^tdW_r,\ea\right.\q t\in[0,T].
\ee
Here the coefficients $b, \si, f$ involve two time variables; the initial condition is a continuous path $\bx\in C([0, T]; \dbR^n)$; and the terminal condition $g$ depends on $t$ as well. A special case of the forward SVIE is the fractional Brownian motion, where $\bx =0$, $b = 0$, $\si = K(t,r)$ for some deterministic kernel $K$. FSVIE has received very strong attention in recent years due to its applications in {\it rough volatility} models, see, e.g., Comte--Renault \cite{CR}, Gatheral--Jaisson--Rosenbaum \cite{GJR-2018}, El Euch--Rosenbaum \cite{Euch, ER-2018}, and Viens--Zhang \cite{Viens-Zhang-2019}. On the other hand, BSVIE has become a popular tool for studying many problems in mathematical finance. For examples, Di Persio \cite{Di Persio-2014} on stochastic differential utility, Yong \cite{Yong 2007}, Wang--Yong \cite{Wang-Yong-2015} and Agram \cite{Agram 2018} on dynamic risk measures, Kromer--Overbeck \cite{Kromer-Overbeck 2017} on dynamic capital allocations, Wang--Sun--Yong \cite{Wang-Sun-Yong-2019} on equilibrium recursive utility and equilibrium dynamic risk measures, to mention a few. More interestingly, in recent years, time-inconsistent problems have attracted many researchers' attention. Among others, the time-inconsistency could be caused by the time-preferences of the decision-makers, which can be described by non-exponential discounting. See the seminal paper by Strotz \cite{Strotz-1955}, and early follow-up works of Pollak \cite{Pollak-1968} and Laibson \cite{Laibson}. For the recent works of time-inconsistent problems relevant to the non-exponential discounting, we mention Karp \cite{Karp}, Ekeland--Lazrak \cite{Ekeland-Lazrak-2010}, Yong \cite{Yong 2012}, Wei--Yong--Yu \cite{Wei-Yong-Yu-2017}, and Hernandez--Possamai \cite{Hernandez-2020}. It is worthy of pointing out that the most suitable dynamic recursive cost functional allowing non-exponential discounting should be described by a BSVIE, as indicated in Wang--Yong \cite{Wang-Yong-2019b}. We remark that the BSVIE in (\ref{FBSVIE1}) is also called type-I BSVIE in the literature. A more general  type-II BSVIE, where $f$ depends not only on $Z_r^t$, but also on $Z^r_t$, appears naturally as an adjoint equation when one studies stochastic maximum principle for controlled FSVIE, see Yong \cite{Yong-2006, Yong-2008}.

Our goal of this paper is to extend the PDE approach to FBSVIEs. This on one hand will help us to understand the structure of FBSVIEs, and on the other hand is helpful for numerical computation of these equations. As mentioned, the PDE approach is based on the flow property of the equations. Unfortunately, due to the two time variable structure, neither FSVIE nor BSVIE satisfies the flow property in the standard sense: for $0\les t <s\les T$,
$$
 X_s\ne X_t+\!\! \int_t^s\! b(s,r,X_r)dr\!+\!\int_t^s\! \si(s,r,X_r)dW_r,\q \ds Y_t\ne Y_s+\!\int_t^s\! f(t,r,X_r,Y_r,Z_r^t)dr-\!\int_t^s\! Z_r^tdW_r.
$$
Our work is built on Viens--Zhang \cite{Viens-Zhang-2019}, Yong \cite{Yong-2016} and Wang--Yong \cite{Wang-Yong-2019}. By introducing  auxiliary two time variable processes $\tilde X^s_t$, $\tilde Y^t_s$, see (\ref{wtX}) and (\ref{wtYZ}) below, \cite{Viens-Zhang-2019} recovers the flow property of the FSVIE in certain sense, and \cite{Yong-2016, Wang-Yong-2019}  recover the flow property of the BSVIE. We remark that in \cite{Viens-Zhang-2019} the backward equation is a standard BSDE, while in \cite{Yong-2016, Wang-Yong-2019} the forward equation is a standard SDE.  Putting together allows us to adopt the PDE approach to FBSVIE (\ref{FBSVIE1}).  We note that the associated PDE will intrinsically depend on the paths of $\tilde X^{[t, T]}_t$, and thus it becomes a {\it path dependent PDE} (PPDE, for short). Then, with a little extra effort, we can actually handle path dependent FBSVIEs, namely $b, \si, f, g$ depend on the paths of $X$, as we will do in the paper. We shall emphasize though, even for the state dependent case (\ref{FBSVIE1}), our results in the paper are new.

To be precise, we shall introduce a two-time variable function $U(t,s, \bx);$ $0\les t\les s\les T$, $\bx\in C([0, T]; \dbR^n)$, which satisfies the PPDE with terminal condition $U(t, T, \bx) = g(t, \bx)$:
\bea
\label{PPDE1}
\left.\ba{c}
\ds\pa_s U(t,s,\bx) + {1\over 2} \big\lan \pa^2_{\bx\bx} U(t,s,\bx), \big(\si^{s,\bx}_{[s, T]}, \si^{s,\bx}_{[s, T]}\big)\big\ran
+\big\lan \pa_{\bx} U(t,s,\bx), b^{s,\bx}_{[s, T]}\big\ran\\
\ns\ds + f\big(t,s, \bx, U(s,s,\bx), \big\lan \pa_{\bx} U(t,s,\bx), \si^{s,\bx}_{[s, T]}\big\ran\big)=0.
\ea\right.
\eea
Here, $\pa_\bx U, \pa^2_{\bx\bx}U$ are the first  order and second order Fr\'echet derivatives with respect to the perturbation of $\bx_{[s, T]}$, and for $\f=b, \si$, $\f^{s,\bx}_{[s, T]}$ refers to the path $\{\f(r, s, \bx)\}_{r\in [s, T]}$. Then we have the following relationship which extends (\ref{Y=U}): denoting  $\hat X^t_r := X_r 1_{[0, t)}(r) +  \tilde X^r_t 1_{[t, T]}(r)$,
\bea
\label{Y=U2}
 Y_t = U\big(t, t,  \hat X^t \big),\q Z^t_s = \big\lan \pa_{\bx} U(t,s, \hat X^s),\, \si^{s, \hat X^s}_{[s, T]}\big\ran,\q\mbox{and}\q \tilde Y^t_s = U(t,s, \hat X^s),
\eea
and similarly we can extend (\ref{U=Y}) to this case, see (\ref{Uts}) below, and thus establish the Feyman--Kac formula for (\ref{PPDE1}). Besides the key flow property, a crucial tool in this analysis is the functional It\^o formula, initiated by Dupire \cite{Dupire-2019} in standard SDE setting and extended to the SVIE setting by \cite{Viens-Zhang-2019}. The PPDE (\ref{PPDE1}) has several important features:

$\bullet$  The state variable $\bx$ has a continuous path on $[0, T]$, and thus is infinite dimensional.

$\bullet$ $U$ depends on two time variables $(t,s)$. In particular, the equation at $(t,s, \bx)$ involves the value $U(s,s, \bx)$, and thus is non-local in the first time variable $t$.

$\bullet$ Alternatively, noting that (\ref{PPDE1}) does not involve derivatives with respect to the first time variable $t$, then we may view $t$ as a parameter instead of a variable. That is, we may view (\ref{PPDE1})  as a system of PPDEs with parameter $t$ and solution  $\{U(t,\cd,\cd)\}_{t\in [0, T]}$. Then this is an (uncountably) infinite dimensional system of PPDEs which are self interacted through the diagonal term $U(s,s, \bx)$.

We next prove the existence of classical solutions to PPDE (\ref{PPDE1}), provided the coefficients are smooth enough in an appropriate sense, and thus establish the above connection between PPDE (\ref{PPDE1}) and FBSVIE (\ref{FBSVIE1}) rigorously. We remark that Peng--Wang \cite{Peng-Wang-2016} obtained the classical solution in the form $u(t, \bx_{[0, t]})$ for a PPDE corresponding to PDE (\ref{PDE1}), associated with the path dependent version of the FBSDE (\ref{FBSDE1}). Our result generalizes  \cite{Peng-Wang-2016} in several aspects.  First, in \cite{Peng-Wang-2016} $u(t, \bx_{[0, t]})$ depends on the path only up to $t$, in particular the path derivative $\pa_\bx u$ there involves only the perturbation of $\bx_t$ and thus is actually a finite dimensional derivative, while our path derivative is indeed a Fr\'{e}chet derivative.  Next, $u$ is finite dimensional, while as mentioned (\ref{PPDE1}) can be viewed as an infinite dimensional system. Moreover, when restricted to the state dependent case, the PPDE in \cite{Peng-Wang-2016} reduces back to the standard PDE (\ref{PDE1}), but (\ref{PPDE1}) has the same features that both the state $\bx$ and the solution $U$ are infinite dimensional.  We also obtain a representation formula for the path derivative $\pa_\bx U(t,s, \bx)$, which is interesting in its own right and is  new in the literature.

The more challenging part is the viscosity solution theory for PPDE (\ref{PPDE1}), in the case that $Y$ is scalar but the coefficients are less smooth. Note that the state space $C([0, T]; \dbR^n)$ is not locally compact, so the standard viscosity solution theory of Crandall--Ishii--Lions \cite{GIL-1992} does not work here.  Moreover, we have some intrinsic adaptiveness requirement on the dependence of the path, which prevents us from applying the viscosity solution theory in infinite dimensional space, see e.g.  Crandall--Lions \cite{Crandall-Lions-1991}, Li--Yong \cite{Li-Yong-1995}, and Fabbri--Gozzi--Swiech \cite{FGS}. One exception in this direction is Ren--Rosestolato \cite{RR}, which however requires some stronger type of regularity and is overall still not satisfactory for our purpose. We shall follow the approach proposed by Ekren--Keller--Touzi--Zhang \cite{EKTZ}, where the pointwise optimization in \cite{GIL-1992} is replaced with an optimal stopping problem under certain nonlinear expectation, and thus the comparison principle can be obtained without requiring the local compactness of the state space. Our PPDE (\ref{PPDE1}) has two major differences from \cite{EKTZ}. First, the nonlinear expectation used in \cite{EKTZ} relies on a family of semi-martinagle measures, while our state process $X$ is not a semi-martingale. Second, the PPDE in \cite{EKTZ} is one dimensional and the comparison principle for classical solutions (if they exist) is quite straightforward, but as mentioned PPDE (\ref{PPDE1}) is  non-local (or viewed as infinite dimensional), and in fact the comparison principle fails in general even for classical solutions. Nevertheless, we shall propose a new notion of viscosity solution to PPDE (\ref{PPDE1}) and establish its well-posedness, including the comparison principle, under an additional assumption  that $f$ is  nondecreasing in $y$. We note that this monotonicity condition is essentially the {\it proper} condition required in \cite{GIL-1992} for elliptic equations. For a standard parabolic equation like (\ref{PDE1}), this condition is redundant because it is implied from the Lipschitz condition by a standard change variable argument. However, the change variable argument fails for (\ref{PPDE1})  because of its non-local structure. We also note that Wang--Yong \cite{Wang-Yong-2015} proved the comparison principle for BSVIEs under the same monotonicity condition. As in the standard literature, since viscosity solution is a local notion (even with some non-local feature here), its comparison principle is much more challenging.

Finally, we investigate briefly two more general FBSVIEs, the coupled FBSVIE (with $b,\si$ depending on $Y$) and the type-II BSVIEs, and  extend the representation formula in these cases.  The more detailed studies on these equations are left for interested readers. We note particularly that our new representation formula for $\pa_\bx U$ relies on a linear type-II BSVIE. For this purpose, we establish a duality result for linear path dependent FSVIE which covers the corresponding results in Yong \cite{Yong-2006, Yong-2008}  and Peng--Yang \cite{Peng-Yang}, and provide an explicit solution for linear BSVIEs which generalizes the result of  Hu--{\O}ksendal \cite{Hu 2018}.

The rest of this paper is organized as follows. In Subsection \ref{sect-literature}, we provide a literature review on the closely related topics. Section \ref{sect-preliminary} collects some preliminary results which will be used in the paper. In Section \ref{sect-PPDE}, we establish the connection between FBSVIEs and PPDEs, and prove the existence of classical solutions under appropriate conditions. Section \ref{sect-viscosity} is devoted to the viscosity solutions of the PPDE.
We extend some results to coupled FBSVIEs and type-II BSVIEs in Section \ref{sect-extension}.
Finally in Section \ref{sect-repre}, we obtain a representation formula for the path derivative $\pa_\bx U(t,s,\bx)$.

\subsection{Literature review on some related topics}
\label{sect-literature}
For FSVIEs, we first refer to Nualart \cite{Nualart-2006} for a comprehensive exposition of fractional Brownian motion, which is a very special case of FSVIEs. In the state dependent case, the well-posedness of FSVIEs can be found in Berger--Mizel \cite{Berger-1980}. Since one cannot apply the Burkholder--Davis--Gundy inequalities for stochastic Volterra integral equations, the well-posedness of path dependent FSVIEs is actually more involved, and we refer to the recent work Ruan--Zhang \cite{Ruan-Zhang-2020}. There has been a growing number of publications on rough volatility models, for which FSVIE is a convenient tool. Besides \cite{CR, Euch, ER-2018,   GJR-2018, Viens-Zhang-2019}, a partial list also includes Abi Jaber-Larsson-Pulido \cite {ALP},  {\color{red}Alos--Leon--Vives \cite{ALV}}, Bayer--Friz--Gatheral \cite{BFG}, Bennedsen--Lunde--Pakkanen \cite{BLP}, Chronopoulou--Viens \cite{CV},   Cuchiero--Teichmann \cite{CT}, Fouque--Hu \cite{FH},  Gatheral--Keller--Ressel \cite{GK},  and  Gulisashvili--Viens--Zhang \cite{GVZ}.

BSVIE was first introduced by Lin \cite{Lin-2002} in a special form. The general form, including type-II BSVIEs, has been studied systematically by Yong \cite{Yong-2006, Yong-2008},  followed by Djordjevic--Jankovic \cite{Djordjevic-Jankovic-2013,Djordjevic-Jankovic-2015}, Shi--Wang--Yong \cite{Shi-Wang-Yong-2015}, Wang--Yong \cite{Wang-Yong-2015},  Wang--Zhang \cite{Wang-Zhang-2017},  Overbeck--Roder \cite{Overbeck-Roder-2018}, Hu--{\O}ksendal \cite{Hu 2018}, Wang--Yong \cite{Wang-Yong-2019},  Popier \cite{Popier-2020}, Hernandez--Possamai \cite{HP2}, to mention a few. In particular, we note that Hamaguchi \cite{Hamaguchi-2020} proved the well-posedness of coupled FBSVIEs over small time horizon. The well-posedness of coupled FBSVIEs over arbitrary time horizon  is still open, to our best knowledge.  We also refer to  \cite{Agram 2018,  Di Persio-2014,  Hernandez-2020, Kromer-Overbeck 2017, Wang-Sun-Yong-2019, Wang-Yong-2019b, Yong 2007} again for some applications of BSVIEs.

The notion of PPDE was first proposed by Peng \cite{Peng-2010}. A crucial tool is the functional It\^o formula, initiated by Dupire \cite{Dupire-2019} and further developed by Cont--Fourni\'{e} \cite{CR1, CR2}. In the semilinear case, Peng--Wang \cite{Peng-Wang-2016} obtained the classical solution and Ekren--Keller--Touzi--Zhang \cite{EKTZ} established the viscosity solution theory. The viscosity solution approach of \cite{EKTZ} has been successfully extended to the fully nonlinear case by Ekren--Touzi--Zhang \cite{ETZ1, ETZ2}, Ren--Touzi--Zhang \cite{RTZ 2017}, and Ren--Rosestolato \cite{RR}. We also refer to Barrasso--Russo \cite{Barrasso-Russo2020}, Cosso--Russo \cite{CR1}, Leao--Ohashi--Simas \cite{LOS}, Lukoyanov \cite{Lukoyanov}, Peng--Song \cite{PS} for some related works, in particular to Cosso--Russo \cite{CR2}, Zhou \cite{Zhou} for some recent interesting developments, and to the book Zhang \cite{Zhang-2017} for more references. We shall remark though that the PPDEs in all the above works are in the semi-martingale setting.  Our PPDE is associated with SVIEs, and the corresponding functional It\^o formula was proved by \cite{Viens-Zhang-2019}. Another closely related work also beyond semi-martinagle setting is Barrasso--Russo \cite{Barrasso-Russo2019}, which studies the so-called decoupled mild solution for a PPDE associated with Gaussian processes.

\section{Preliminaries}\label{sect-preliminary}

Let $T>0$ be the time horizon,  $\Omega:=C([0,T];\dbR^d)$ the canonical space, $W$ the canonical process (namely $W(\omega) = \omega$), $\dbP$ the Wiener measure (namely $W$ is a standard $d$-dimensional Brownian motion under $\dbP$), and $\dbF:= \dbF^W$ augmented with the $\dbP$-null sets. Denote
$$\ba{c}
\ns\ds\dbT=[0,T],\qq\dbT^2=[0,T]\times[0,T],\\
\ns\ds\dbT^2_-=\big\{(t,s)\bigm|0\les s\les t\les T\big\},\qq\dbT^2_+=\big\{(t,s)\bigm| 0\les t\les s\les T\big\}.\ea$$
Here ``$_-$" indicates the left neighborhood of $t$, and ``$_+$" indicates the right neighborhood of $t$. For any Euclidean space $\dbH$ (say, $\dbR^n$, $\dbR^{m\times d}$, etc.), let
$$
\dbL^p_\dbF(0,T;\dbH)=\Big\{\f:[0,T]\times \Omega\to\dbH\bigm|\f\hb{ is $\dbF$-progressively measurable, }\dbE\int_0^T|\f(s)|^pds<\infty\Big\}.$$
Our state space is $\dbX:=C([0,T];\dbR^n)$, equipped  with the uniform norm:
\bel{uniform}
\|\bx\|=\sup_{t\in[0,T]}|\bx_t|,\qq\forall\bx\in\dbX.
\ee

In this section we review and present some basic results concerning forward and backward SVIEs,
including a continuous-norm estimate for the adapted solution to a class of BSVIEs.
Moreover, among other things, we shall introduce two auxiliary processes $\wt X$ and $\wt Y$ so that the flow property of the adapted solutions can be established in an extended sense. It turns out that such a property will play an essential role in proving the relation between FBSVIEs and PPDEs.

Before going further, we make a convention which will be used in the rest of the paper.
For any map $\f:\dbT^2\times\dbX\times\dbH\times\Omega\to\wt\dbH$,
where $\dbH$ and $\wt\dbH$ are any Euclidean spaces (could be $\dbR^m$, $\dbR^m\times\dbR^{m\times d}$, etc.),
we simply say that $\f$ is {\it progressively measurable} if
\bel{f}\f(t,r,\bx,h,\omega)=\f(t,r,\bx_{r\land\cd},h,\omega_{r\land\cd}),\q\forall(t,r,\bx,h,\omega)
\in\dbT^2 \times \dbX\times\dbH\times\Omega,\ee
and the above map is measurable. In the above $\dbT^2$ can be replaced by $\dbT^2_\pm$; also some independent variables can be absent.

\subsection{The well-posedness and flow property of FSVIEs}\label{sect-Fflow}

Given $\bx\in \dbX$, consider an FSVIE:
\bea
\label{FSVIE}
X_t = \bx_t + \int_0^t b(t,r, X_\cd) dr + \int_0^t \si(t,r, X_\cd) dW_r,\q t\in\dbT.
\eea
We shall assume the following.
\begin{assumption}
\label{assum-FSVIE}\rm
The map $(b,\si):\dbT^2_-\times\dbX\to\dbR^n\times\dbR^{n\times d}$ is progressively measurable satisfying:
\begin{enumerate}[(i)]
\item For some constant $C_0>0$, $|b(t,r,\b0)|+|\si(t,r,\b0)|\les C_0$ for all $(t,r)\in\dbT^2_-.$

\item The map $\bx\mapsto(b(t,r,\bx),\si(t,r,\bx))$ is uniformly Lipschitz continuous under the norm $\|\cd\|$.

\item The map $t\mapsto(b(t,r,\bx),\si(t,r,\bx))$ is differentiable, with $\pa_tb$ and $\pa_t\si$ also satisfying the conditions as in (i) and (ii). \end{enumerate}
\end{assumption}

\ms

The following result follows from Ruan--Zhang \cite{Ruan-Zhang-2020}.

\begin{proposition}
\label{prop-fSVIE} \sl
Under  {\rm Assumption \ref{assum-FSVIE}}, FSVIE \rf{FSVIE} admits a unique strong solution $X$ such that $X$ is continuous in $t$ and the following estimate holds true: for any $p>1$,
\bel{fSVIEest}
\dbE\big[\|X\|^p\big]\les C_p\big[1+\|\bx\|^p\big].\ee
\end{proposition}

We remark that due to the first time variable $t$ in $\si$,
one cannot directly apply the Burkholder--Davis--Gundy inequalities in the Volterra setting. Assumption \ref{assum-FSVIE} (iii) helps us to get around that. In the state dependent case: $\si = \si(t, r, X_r)$ ($b$ can be path dependent, although in the literature, typically, it is also state dependent), the well-posedness of  \rf{FSVIE}  follows from standard arguments, see e.g. Berger--Mizel \cite{Berger-1980}. The pathwise continuity of $X$ as well as the norm estimate \rf{fSVIEest} hold true for $b, \si$ satisfying weaker continuity in the spirit of \rf{Holderf} below. The arguments are similar to those of  Proposition \ref{prop-BSVIE} below and we skip the details. It will be interesting to see if it is possible to weaken
Assumption \ref{assum-FSVIE} (iii) in the path dependent case.

Note that  $X$ is neither a Markov process nor a semimartingale. Even worse, in general the flow property fails in the following sense: for fixed $t$,
\bea
\label{flowX1}
X_s \neq X_t + \int_t^s b(s,r, X_\cd) dr + \int_t^s \si(s,r, X_\cd) dW_r,\q s\in (t, T].
\eea
One may refer to this as the {\it time-inconsistency}. To overcome this deficiency,  Viens--Zhang \cite{Viens-Zhang-2019} introduced an auxiliary  process with two time variables:
\bea
\label{wtX}
\wt X^s_t = \bx_s + \int_0^t b(s,r, X_\cd) dr + \int_0^t \si(s,r, X_\cd) dW_r,\q (t,s)\in\dbT^2_+.
\eea
This process enjoys the following nice properties:

\ms

$\bullet$ For fixed $s$, the process $[0,s]\ni t\mapsto\wt X^s_t$ is an $\dbF$-semimartingale with $\wt X^t_t=X_t$;

\ss

$\bullet$ For fixed $t$, the process $[t,T]\ni s\mapsto\wt X^s_t $ is $\cF_t$-measurable and  continuous;

\ss

$\bullet$ The {\it flow property} holds in the following sense: for any $\dbF$-stopping time $\t$,
\bea
\label{flowX2}
X_s = \wt X^s_\t  + \int_\t^s b(s,r, X_\cd) dr + \int_\t^s \si(s,r, X_\cd) dW_r,\q s\in [\t, T].
\eea

We remark that, in the state dependent case as in (\ref{FBSVIE1}), \rf{flowX2} implies
$$
X_s = \wt X^s_t  + \int_t^s b(s,r, X_r) dr + \int_t^s \si(s,r, X_r) dW_r,\q (t,s)\in\dbT^2_+.
$$
One can easily see that, conditional on $\wt X^{[t,T]}_t$, $X_{[0, t)}$ and $X_{(t, T]}$ are conditionally independent. So this can be viewed as a generalized Markov property.

\subsection{The well-posedness and flow property of BSVIEs}
\label{sect-Bflow}

Consider the following path dependent BSVIE:
\bea
\label{BSVIE}Y_t=g(t,X_\cd)+\int_t^Tf(t,r,X_\cd,Y_r,Z^t_r)dr-\int_t^TZ^t_rdW_r,\q t\in \dbT,
\eea
where $Y$ is $m$-dimensional and hence $Z$ is $(m\times d)$-dimensional. We shall assume
\begin{assumption}
\label{assum-BSVIE}\rm
The map $f:\dbT^2_+\times\dbX\times\dbR^m\times\dbR^{m\times d}\to\dbR^m$ is progressively measurable and the map $g:\dbT\times\dbX\to\dbR^m$ is $\cF_T$-measurable satisfying:

\begin{enumerate}[(i)]

\item For some constant $C_0>0$, it holds
$$|f(t,r,\bx,0,0)|+|g(t,\bx)|\les C_0\big[1+\|\bx\|\big],\qq\forall(t,r,\bx)\in\dbT^2_+\times\dbX;$$

\item The map $(y,z)\mapsto f(t,r,\bx,y,z)$ is uniformly Lipschitz continuous;

\item The map $t\mapsto(f(t,r,\bx,y,z),g(t,\bx))$ is locally uniformly continuous in the following sense: for some modulus of continuity function $\rho$,
\bel{Holderf}
\begin{aligned}
&|f(t-\d,r,\bx,y,z)-f(t,r,\bx,y,z)|+|g(t-\d,\bx)-g(t,\bx)| \\
&\q\les C\big[1+\|\bx\|+|y|+|z|\big]\rho(\d),\q\forall (t,r,\bx,y,z)\in\dbT^2_+\times\dbX\times \dbR^m\times\dbR^{m\times d},\,\d\in[0,t].
\end{aligned}
\ee
\end{enumerate}
\end{assumption}

\begin{proposition}
\label{prop-BSVIE} \sl
Under  {\rm Assumptions \ref{assum-FSVIE}} and {\rm \ref{assum-BSVIE}}, BSVIE \rf{BSVIE} admits a unique strong solution $(Y, Z)$ such that $Y$ is continuous in $t$ and the following estimate holds true: for each $p>1$,
\bel{BSVIEest}
\dbE\Big[\sup_{0\les t\les T} |Y_t|^p \Big] + \sup_{0\les t\les T}\dbE\Big[\Big(\int_t^T |Z^t_s|^2ds\Big)^{{p\over2}}\Big]
\les C_p\big[1+\|\bx\|^p\big].
\ee
\end{proposition}

The proof of the well-posedness of BSVIE (\ref{BSVIE}) is standard and could be found in Yong \cite{Yong-2008}, where the pathwise continuity of $Y$ was proved for a more general BSVIE, but under much stronger technical conditions. Our arguments for the pathwise continuity and the above estimate \rf{BSVIEest} seems to be new in the literature. We note that (\ref{Holderf}) is much weaker than Assumption \ref{assum-FSVIE} (iii),  because  $f$  is state dependent on $(Y, Z)$. To facilitate the proof, we introduce the following standard BSDE parameterized by $t\in\dbT$ with adapted solution $(\wt Y^t,\wt Z^t)$:
\bel{wtYZ}
\wt Y^t_s=g(t,X_\cd)+\int_s^Tf(t\land r,r,X_\cd,Y_r,\wt Z^t_r)dr-\int_s^T\wt Z^t_r dW_r,\q s\in[0,T].
\ee

\begin{proof}
First, from \cite{Yong-2008} we know \rf{BSVIE} admits a unique solution $(Y, Z)$ such that $Y\in \dbL_\dbF^2(0,T;\dbR^m)$ and $Z^t\in \dbL^2_\dbF(t,T;\dbR^{m\times d})$. Compare \rf{BSVIE} with the following linear BSDE on $[t,T]$:
$$\h Y^t_s=g(t,X_\cd)+\int_s^Tf(t\land r,r,X_\cd,Y_r,Z^t_r)dr-\int_s^T\h Z^t_rdW_r,\q s\in [t,T].$$
It is obvious that $\h Y^t_t=Y_t$ and $\h Z^t_r=Z^t_r$. This implies that $(\h Y^t,\h Z^t)$ satisfies \rf{wtYZ}. Then we have
\bel{BSDE-BSVIE-equal}
\wt Y^t_t  = Y_t,\q \wt Z^t_r = Z^t_r,\q (t,r)\in\dbT^2_+.
\ee

Next, by \rf{fSVIEest}, \rf{Holderf} and the standard BSDE arguments, we have
\bel{wtYest}
\left.\ba{c}
\displaystyle\sup_{0\les t\les T} \dbE\Big[ \sup_{0\les s\les T} |\wt Y^{t}_s|^p + \Big(\int_0^T |Z^t_s|^2ds\Big)^{p\over 2}\Big] \les C_p\big[1+\|\bx\|^p\big];\ms\\
\displaystyle  |\wt Y^{t}_s - \wt Y^{t'}_s| \les C_p\Big[1+\big(\dbE_s[\|X\|^p]\big)^{1\over p}\Big] \rho(|t-t'|),~\as,\q \forall t, t', s\in \dbT.
\ea\right.
\ee
 Note that $\wt Y_s^{t}$, $\wt Y_s^{t'}$, and $\dbE_s[\|X\|^p]$ are pathwise continuous in $s$, then we have
\bea
\label{wtYholder}
\sup_{s\in\dbT} |\wt Y^{t}_s - \wt Y^{t'}_s| \les C_p\Big[1+\sup_{s\in\dbT} \big(\dbE_s[\|X\|^p]\big)^{1\over p}\Big]\rho(|t-t'|),~ \as,\q\forall t, t'\in \dbT.
\eea
Note that, by (standard) Doob's maximum inequality,
\beaa
\dbE\Big[\sup_{s\in \dbT} \dbE_s[\|X\|^p] \Big] \les C_p \Big(\dbE[\|X\|^{2p}]\Big)^{1\over 2} \les C_p\big[1+\|\bx\|^p\big]<\infty.
\eeaa
Let $\{t_i\}_{i\ges 1}$ be the rationals in $[0, T]$.
There exits an $\Omega_1\subset\Omega$ such that $\dbP(\Omega_1)=1$, $\wt Y^{t_i}_s(\omega)$ is continuous in $s$, and
\bea
\label{wtYholder1}
\left.\ba{c}
\ds  \sup_{s\in\dbT} |\wt Y^{t_i}_s - \wt Y^{t_j}_s|(\omega) \les  C_p(\omega)\rho(|t_i-t_j|),\q \forall (i, j),~ \forall \omega\in \Omega_1,\\
\ds \mbox{where}\q C_p(\omega):= \Big[1+\sup_{s\in\dbT} \big(\dbE_s[\|X\|^p]\big)^{1\over p}\Big](\omega) <\infty,\q\forall \omega\in \Omega_1.
\ea\right.
\eea
For any $t\in \dbT$, by \rf{wtYholder}, there exists an $\Omega^t\subset\Omega$ such that $\dbP(\Omega^t)=1$ and
\bea
\label{wtYholder2}
\sup_{s\in\dbT} |\wt Y^{t}_s - \wt Y^{t_j}_s|(\omega) \les C_p(\omega)\rho(|t-t_j|),\q \forall j,~ \forall \omega\in \Omega^t\cap \Omega_1.
\eea
For any $(t,s)\in \dbT^2$, we define
\bea
\label{barYt}
\bar Y^{t}_s (\omega) :=\limsup_{t_j \to t} \wt Y^{t_j}_s(\omega),\q  \omega\in\Omega.
\eea
By \rf{wtYholder1} we see that the above $\limsup$ is actually a limit  for $\omega\in \Omega_1$. Then,  for any $ \omega\in \Omega_1$,
\bea
\label{wtYholder3}
\bar Y^{t}_s(\omega)~\mbox{is continuous in $s$, }\q \sup_{s\in\dbT} |\bar Y^{t}_s - \bar Y^{t'}_s|(\omega) \les  C_p(\omega)\rho(|t-t'|),\q \forall t, t'\in \dbT.
\eea
So $\bar Y$ is (uniformly) continuous in $(t,s)\in \dbT^2$ for all $\omega\in \Omega_1$.  Moreover, by \rf{wtYholder2} we have
\bea
\label{wtYholder5}
\wt Y^{t}_s(\omega)=\bar Y^t_s(\omega),\q \forall s\in \dbT,~ \forall \omega\in\Omega^t\cap \Omega_1.
\eea
Since $\dbP(\Omega^t\cap \Omega_1)=1$, so $\bar Y$ is a desired version of $\wt Y^t_s$,
and thus, by always considering this version,  $\wt Y^t_s$ is jointly continuous in $(t,s)$, a.s.
In particular, this implies that $Y_t = \wt Y^t_t$ is continuous in $t$, a.s.

Finally, applying the standard BSDE estimates on \rf{wtYZ} we have
$$
|Y_t|^p=|\wt Y^t_t|^p\les C_p\dbE_t\Big[1+\|X\|^p+\int_t^T|Y_r|^pdr\Big],\q\as
$$
Since $t\mapsto Y_t$ is continuous almost surely, we obtain from the Doob's maximum inequality that
$$\dbE\Big[\sup_{0\les t\les T}|Y_t|^p\Big]\les C_p\Big(\dbE\Big[1+\|X\|^{2p}+\int_0^T|Y_r|^{2p} dr\Big]\Big)^{1\over 2}\les C_p\big[1+\|\bx\|^p\big],$$
where the second inequality thanks to \rf{fSVIEest} and the first line of \rf{wtYest}. This, together with the first line of \rf{wtYest} again, implies \rf{BSVIEest}.
\end{proof}

\ms

Similar to the forward case, in general the flow property fails in the following sense: for fixed $s$,
\bel{flowY1}
Y_t\ne Y_s+\int_t^sf(t,r,X_\cd,Y_r,Z^t_r)dr-\int_t^sZ^t_rdW_r,\q t\in[0,s).\ee
However, we may recover the flow property by utilizing the auxiliary process $\wt Y$:
\bel{flowY2}
Y_t=\wt Y^t_s+\int_t^s f(t,r,X_\cd,Y_r,Z^t_r)dr-\int_t^sZ^t_rdW_r,\q t\in[0,s].
\ee

\subsection{The functional It\^{o} formula}
The materials in this subsection follow from Viens--Zhang \cite{Viens-Zhang-2019}. Recall $\dbX:= C(\dbT;\dbR^n)$ and define
\bel{XX}\ba{c}
\ns\ds\dbX_t:= C([t,T];\dbR^n);\\
\ns\ds\h\dbX:=D(\dbT;\dbR^n)\equiv\Big\{\bx:\dbT\to\dbR^n\bigm|\bx\hb{ is right-continuous with left-limits}\Big\}.
\ea\ee
Clearly, $\dbX$ is a subset of $\h\dbX$. Also, hereafter, for any $\eta\in\dbX_t$, we automatically extend it to be zero on $[0,t)$, still denote it by $\eta$. Then $\dbX_t\subseteq\h\dbX$. Next, we define
$$
\ds\L:=\dbT\times\dbX,\q\h\L:=\Big\{(t,\bx)\in\dbT\times\h\dbX:\bx|_{[t,T]}\in\dbX_t\Big\},\q
\ds{\bf d}\big((t,\bx),(t',\bx')\big):=|t-t'|+\|\bx-\bx'\|,
$$
with $\|\bx\|=\sup_{t\in\dbT}|\bx_t|$ for $\bx\in \h\dbX$.
It can be shown that ${\bf d}$ is a metric under which $\h\L$ is a complete metric space. Now, let $C^0(\h\L)$ denote the set of all functions $u:\h\L\to\dbR$ which are continuous under ${\bf d}$. For any $u\in C^0(\h\L)$ and given $(t,\bx)\in\h\L$, define
\bel{u_t}\pa_tu(t,\bx)=\lim_{\d\da0}{u(t+\d,\bx)-u(t,\bx)\over\d},
\ee
provided the limit exists, and define $\pa_\bx u(t,\bx)$ as the Fr\'{e}chet derivative with respect to $\bx|_{[t, T]}$, namely $\pa_\bx u(t,\bx):\dbX_t\to\dbR$ is the linear  functional satisfying the following:
\bel{u_o}u(t,\bx+\eta)-u(t,\bx)=\lan\pa_\bx u(t,\bx),\eta\ran+o(\|\eta\|),\qq\forall\eta\in\dbX_t.\ee
It is clear that this is equal to the G\^{a}teux derivative:
\bel{gateaux}\lan\pa_\bx u(t,\bx),\eta\ran=\lim_{\e\to0}{u(t,\bx+\e\eta)-u(t,\bx)\over\e},\qq\forall\eta\in\dbX_t.
\ee
Similarly we define the second order derivative $\pa^2_{\bx\bx}u(t,\bx)$ as a bilinear functional on $\dbX_t\times \dbX_t$:
\bel{u_oo}\lan\pa_\bx u(t,\bx+\eta),\eta'\ran-\lan\pa_\bx u(t,\bx),\eta'\ran=\lan\pa^2_{\bx\bx}u(t,\bx),(\eta,\eta')\ran+o(\|\eta\|),\qq\forall
\eta,\eta'\in\dbX_t.
\ee

\ms
\bde{hatC12} \rm
Let $C^{1,2}_+(\h\L)$ denote the set of  $u\in C^0(\h\L)$ such that $\pa_tu,\pa_\bx u,\pa_{\bx\bx}^2u$ exist on $\dbT\times\h\dbX$ and satisfy:
\begin{enumerate}[(i)]
\item There exist constants $\k>0$ and $C>0$ such that, for any $(t, \bx)$,
\beaa
|\pa_t u(t, \bx)|+\!\!\!\! \sup_{\eta\in \dbX_t,\, \|\eta\|\les 1} \!\!\!\!  |\lan \pa_\bx u(t, \bx), \eta\ran| +\!\!\!\! \sup_{\eta, \eta'\in \dbX_t,\, \|\eta\|, \|\eta'\|\les 1}\!\!\!\!
 |\lan \pa^2_{\bx\bx} u(t,\bx), (\eta, \eta')\ran| \les C[1+\|\bx\|^\k].
\eeaa
\item For any $\eta, \eta'\in \dbX$, $\pa_t u(t, \bx)$, $\lan \pa_\bx u(t, \bx), \eta|_{[t, T]}\ran$, $\lan \pa_{\bx\bx} u(t, \bx), (\eta|_{[t, T]}, \eta'|_{[t, T]})\ran$ are continuous in $(t, \bx)$, where the continuity in $t$ always means right-continuity.

\item There exist $\k>0$ and a modulus of continuity function $\rho$ such that:
$$
\big|\lan \pa^2_{\bx\bx} u(t, \bx) -  \pa^2_{\bx\bx} u(t, \bx'), (\eta, \eta)\ran\big| \les \big[1+\|\bx\|^\k + \|\bx'\|^\k\big] \|\eta\|^2 \rho(\|\bx-\bx'\|).
$$
\end{enumerate}
\ede
We remark that the function $u$ will be involved in some backward equations, so both in \rf{u_t} and in Definition \ref{hatC12} (ii) the time regularity is only required to be from right. We note that \cite{Viens-Zhang-2019} assumes $\lan \pa_\bx u(t, \bx), \eta|_{[t, T]}\ran$ etc is continuous in $t$, but actually it can only be right continuous because of the indicator function in $\eta|_{[t, T]}$ and in all the arguments in \cite{Viens-Zhang-2019} only right continuity is used.  For any $u_1,u_2\in C^{1,2}(\h\L)$, if $u_1=u_2$ on $\L$, by \cite{Viens-Zhang-2019} we have, for any $(t,\bx)\in \L$ and $\eta\in \dbX_t$,
\bel{pauequal}
\left.\ba{c}
\pa_tu_1(t,\bx)=\pa_tu_2(t,\bx),\q\lan\pa_\bx u_1(t,\bx),\eta\ran=\lan\pa_\bx u_2(t,\bx),\eta\ran,\\
\ns \lan\pa^2_{\bx\bx}u_1(t,\bx),(\eta,\eta)\ran=\lan\pa^2_{\bx\bx}u_2(t,\bx),(\eta,\eta)\ran.
\ea\right.
\ee

\ms
\bde{C12} \rm Let $C^{1,2}_+(\L)$ be the set of functions $u:\L\to\dbR$ such that there exists a $\h u\in C^{1,2}_b(\h\L)$ satisfying $u=\h u$ on $\L$. For such a case, define
$$
\pa_tu(t,\bx)=\pa_t\h u(t,\bx),\q\pa_\bx u(t,\bx)=\pa_\bx\h u(t,\bx),\q \pa_{\bx\bx}^2u(t,\bx)=\pa_{\bx\bx}^2\h u(t,\bx),\qq\forall (t, \bx)\in\L.
$$
\ede

We emphasize that, by \rf{pauequal}, $\pa_{\bx\bx}^2u(t,\bx)$ is well defined (or say independent of the choice of $\h u$) only on $(\eta, \eta)$, rather than on general $(\eta, \eta')$. However, this is sufficient for our purpose.

\ms

Define, for $\bx\in \dbX,\, \eta \in \dbX_t$, $\f: \dbT_-^2 \times \dbX \to \dbR^k$ with appropriate dimension $k$,
\bel{oplus}
(\bx\oplus_t\eta)(s):=\bx_s{\bf1}_{[0,t)}(s)+\eta_s{\bf1}_{[t,T]}(s),~ s\in \dbT;\qq \f^{t, \bx}_s := \f(s, t, \bx),~ s\in [t, T].\ee
The main result of \cite{Viens-Zhang-2019} is the following functional It\^{o} formula.

\begin{proposition}\label{Ito} \sl
 Suppose  {\rm Assumption \ref{assum-FSVIE}} holds. Let $X$ be the solution to FSVIE \rf{FSVIE}, $\wt X$ be the auxiliary process defined by \rf{wtX}, and $u\in C_+^{1,2}(\L)$. Then, viewing $\wt X_t(\omega) \in \dbX_t$,
\bel{Ito-formula}
\left.\ba{c}
\ds du(t,\h X^t)=\Big[\pa_tu(t,\h X^t)
+{1\over2}\big\lan \pa^2_{\bx\bx}u(t,\h X^t),\,(\si^{t,X},\si^{t,X})\big\ran +\big\lan\pa_\bx u(t,\h X^t),b^{t,X}\big\ran\Big] dt\\
\ns\ds +\big\lan\pa_\bx u(t,\h X^t),\si^{t,X}\big\ran dW_t,\qq\mbox{where}\q  \h X^t := X\1n\oplus_t\wt X_t.
\ea\right.
\ee
\end{proposition}

\subsection{FBSVIEs with random coefficients}
For later purpose, we shall consider a more general FSVIE with random coefficients:
\bel{FSVIE2}
\check X_t = \bx_t + \int_0^t \check b(t,r, \omega, \check X_\cd) dr + \int_0^t \check\si(t,r, \omega, \check X_\cd) dW_r,\q t\in\dbT.\ee

\begin{assumption}
\label{assum-FSVIE2}\rm Let  $(\check b, \check \si): \dbT^2_- \times\Omega\times \dbX \to \dbR^{n} \times \dbR^{n \times d}$ be progressively measurable satisfying:
\begin{enumerate}[(i)]

\item The map $\bx\mapsto(\check b(t,r,\omega,\bx),\check\si(t,r,\omega,\bx))$ is uniformly Lipschitz continuous under the norm $\|\cd\|$.

\item{} The map $t\mapsto(\check b(t,r,\omega,\bx),\check\si(t,r,\omega,\bx))$ is differentiable with $(\pa_t b, \pa_t\si)$ also satisfing (i), and
\bel{Ip1}
I^p_p:= \sup_{(t, r) \in \dbT^2_-} \dbE\Big[|\check b(t,r, \cd, \b0)|^p +|\check \si(t,r, \cd, \b0)|^p+|\pa_t \check b(t,r, \cd, \b0)|^p +|\pa_t \check \si(t,r, \cd, \b0)|^p \Big]<\infty.
\ee
\end{enumerate}
\end{assumption}

We have the following result, also due to \cite{Ruan-Zhang-2020}.

\begin{proposition}
\label{prop-FSVIE2} \sl
Under {\rm Assumption \ref{assum-FSVIE2}}, FSVIE \rf{FSVIE2} admits a unique strong solution $\check X$ such that  $\check X$ is continuous in $t$ and the following estimate holds true:
\bea
\label{FSVIEest2}
\dbE[\|\check X\|^p] \les C_p \big[\|\eta\|^p + I_p^p\big],\q\mbox{where $I_p^p$ is defined in \rf{Ip1}}.
\eea
\end{proposition}

Similarly, we consider a more general BSVIE with random coefficients:
\bea
\label{BSVIE2}
\check Y_t = \check g(t,\omega) + \int_t^T\check  f(t,r, \omega,  \check Y_r, \check Z^t_r) dr - \int_t^T \check Z^t_r dW_r,\q t\in \dbT.
\eea

\begin{assumption}
\label{assum-BSVIE2}\rm The map
$\check f:\dbT^2_+\times\Omega\times\dbR^m\times\dbR^{m\times d}\to\dbR^m$ is progressively measurable and the map $\check g:\dbT\times\Omega\to\dbR^m$ is $\cF_T$-measurable satisfying:
\bel{Ip2}
\check I_p^p:=\sup_{t\in\dbT}\dbE\Big[\Big(\int_t^T|\check f(t,r,\cd,0,0)|dr\Big)^p +|\check g(t,\cd)|^p\Big]<\infty,\ee
and the map $(y,z)\mapsto\check f(t,r,\omega,y,z)$ is uniformly Lipschitz continuous.
\end{assumption}
From Shi--Wang--Yong \cite{Shi-Wang-Yong-2015}, we have the following standard result.

\begin{proposition}
\label{prop-BSVIE2} \sl
Under  {\rm Assumption \ref{assum-BSVIE2}}, BSVIE \rf{BSVIE2} admits a unique strong solution $(\check Y, \check Z)$ such that
the following estimate holds true:
\bel{BSVIE2est}
 \sup_{0\les t\les T}\dbE\Big[ |\check Y_t|^p+ \Big(\int_t^T |\check Z^t_s|^2ds\Big)^{{p\over2}}\Big]
\les C_p\check I_p^p,\q\mbox{where $\check I_p^p$ is defined in \rf{Ip2}}.
\ee
\end{proposition}

Another important property of BSVIEs is the following comparison principle, due to Wang--Yong \cite{Wang-Yong-2015}. For $y,\tilde y\in\dbR^m$, we say $y\les\tilde y$ if $y_i\les\tilde y_i$, $i=1,\cds,m$.

\begin{proposition}\label{prop-comparison} \sl
For $k=1,2$, let $\check f^k, \check g^k$ satisfy {\rm Assumption \ref{assum-BSVIE2}} and $( \check Y^k, \check Z^k)$ be the solutions to the corresponding BSVIE \rf{BSVIE2}.
Assume $\check f^1 \les \check f^2$ and $\check g^1 \les \check g^2$. Assume further that, either for $k=1$ or $k=2$, $\check f^k$ is nondecreasing in $y$ (in the componentwise sense), and $\check f^k_i$ does not depend on $z_j$, for $i\neq j$, where $z_j\in \dbR^d$ is the $j$-th row of $z\in \dbR^{m\times d}$.  Then we have $\check Y^1_t \les \check Y^2_t$, $0\les t\les T$, a.s.
\end{proposition}

\section{The Path Dependent Feynman--Kac Formula}\label{sect-PPDE}
\setcounter{equation}{0}

In this section, we are going to establish the relations between PPDEs and FBSVIEs.

\subsection{From PPDE to FBSVIE}

Recall the FSVIEs \rf{FSVIE}--\rf{wtX} and BSVIEs \rf{BSVIE}--\rf{wtYZ}.
Recall \rf{PPDE1}, let us introduce the following system of PPDEs:
%
\bel{UPPDE}\left\{\1n\ba{lll}
\ds\cL U(t,s,\bx):=\pa_sU(t,s,\bx)+{1\over2}\lan\pa^2_{\bx\bx}U(t,s,\bx),(\si^{s,\bx}, \si^{s,\bx})\ran+\lan\pa_\bx U(t,s,\bx),b^{s,\bx}\ran\\
\ns\ds\qq\qq\qq+f\big(t,s,\bx,U(s,s,\bx),\lan\pa_\bx U(t,s,\bx),\si^{s,\bx}\ran\big)=0,
\q(t,s,\bx)\in\dbT^2_+\times\dbX,\\
\ns\ds U(t,T,\bx)=g(t,\bx),\qq (t,\bx)\in[0,T]\times\dbX,
\ea\right.\ee
where, for $\f=b,\si$, $\f^{s,\bx}$ is defined by \rf{oplus}. As we see in \rf{Ito-formula}, $\bx$ will correspond to $X\oplus_s\wt X_s$, rather than $X$.
However, due to the required adaptiveness, one has
$$\f^{s, X\oplus_s \wt X_s}_r = \f(r, s, X\oplus_s \wt X_s) = \f(r, s, X) = \f^{s,X}_r,~~ r\in [s, T],\q \hbox{for}~ \f=b,\si.$$
We emphasize that the derivatives in $\cL U(t,s,\bx)$ are with respect to $(s,\bx)$ only. As mentioned in Introduction, since \rf{UPPDE} involves the diagonal value $U(s,s,\bx)$, it is non-local in the first time variable $t$. Alternatively, if we view $t$ as a parameter rather than an independent variable, then \rf{UPPDE} is an (uncountably) infinite dimensional PPDE system self-interacted through $U(s,s,\bx)$.

We call $U\in C^0(\dbT^2_+\times\dbX)$ a {\it classical solution} to the PPDE \rf{UPPDE}
if $U(t, \cd) \in C^{1,2}_+([t, T]\times \dbX)$ for all $t\in \dbT$,
where $C^{1,2}_+([t, T]\times \dbX)$ is defined in the spirit of Definitions \ref{hatC12} and \ref{C12}, but restrict to $s\in [t, T]$ only, and \rf{UPPDE} is satisfied in the classical sense.

\begin{theorem}
\label{thm-FK1} \sl
Under {\rm Assumptions \ref{assum-FSVIE}} and {\rm\ref{assum-BSVIE}},
if the PPDE \rf{UPPDE} has a classical solution $U$, then, for any $(t,s)\in \dbT^2_+$ and recalling the $\h X$ in \rf{Ito-formula},
\bel{YU}
\wt Y^t_s = U(t, s, \h X^s),\q Y_t = U(t, t, \h X^t),\q Z^t_s
= \wt Z^t_s=\lan\pa_\bx U(t, s, \h X^s),\,\si^{s,X}\ran.
\ee
\end{theorem}

\begin{proof}
Fix $t$. Applying the  functional It\^{o} formula \rf{Ito-formula} to $U(t,\cd,\h X^\cd)$, we have
$$\ba{ll}
\ns\ds
dU(t,s,\h X^s)=\pa_s U(t,s,\h X^s)ds+\lan\pa_\bx U(t,s,\h X^s),b^{s,X}\ran ds\\
\ns\ds\qq\qq\qq+{1\over2}\lan \pa^2_{\bx\bx}U(t,s,\h X^s),\,(\si^{s,X},\si^{s,X})\ran ds+\lan\pa_\bx U(t,s,\h X^s),\si^{s, X}\ran dW_s.
\ea$$
Since $U$ satisfies the PPDE \rf{UPPDE}, the above implies that
\bel{proof-FK}
dU(t,s,\h X^s)
=-f\big(t, s, X, U(s,s, \h X^s), \lan\pa_\bx U(t,s,\h X^s),\si^{s,X}\ran\big) ds +\lan\pa_\bx U(t,s,\h X^s),\si^{s, X}\ran dW_s.
\ee
Note that $U(t,T,X)=g(t,T,X)$, integrating \rf{proof-FK} over $[t, T]$ we have:
\begin{align*}
U(t,t,\h X^t)
&=g(t,T,X)+\int_t^Tf\big(t, r, X, U(r,r, \h X^r), \lan\pa_\bx U(t,r,\h X^r),\si^{r,X}\ran\big) dr\\
& \q-\int_t^T\lan\pa_\bx U(t,r,\h X^r),\si^{r, X}\ran dW_r.
\end{align*}
That is, $(\h Y_t,\, \h Z^{t}_s):=(U(t,t,\h X^t),\,\lan\pa_\bx U(t,s,\h X^s),\si^{s,X}\ran)$ satisfies BSVIE \rf{BSVIE}.
Then, from the uniqueness of BSVIEs we obtain $\h Y=Y$ and $\h Z=Z$, hence the last two formulae in \rf{YU}.
Moreover, by substituting these into \rf{proof-FK}, we have
$$
dU(t,s,\h X^s) =-f(t, s, X, Y_s, Z^t_s) ds +Z^t_s dW_s.
$$
This clearly implies the first formula in \rf{YU}.
\end{proof}

\subsection{From FBSVIE to PPDE }
In this subsection we proceed with the opposite direction: constructing $U(t,s, \bx)$ by using FBSVIEs.
We emphasize again that the $\bx$ here corresponds to $X\oplus_s \wt X_s$.

First, for any $(s,\bx)\in \L$, denote $X^{s,\bx}_l:=\bx_l$, $l\in[0,s]$, and consider the following  FBSVIE:
\bel{FBSDEth}\left.\1n\ba{ll}
\ds X^{s,\bx}_l=\bx_l+\int_s^l  b(l ,r,X_\cd^{s,\bx})dr+\int_s^l \si(l ,r, X_\cd^{s,\bx})dW_r;\\
\ns\ds Y^{s,\bx}_l =g(l ,X_\cd^{s,\bx})+\int_l ^Tf(l ,r,X_\cd^{s,\bx},Y^{s,\bx}_r,
Z^{l ,s,\bx}_r)dr-\int_l ^TZ^{l ,s,\bx}_rdW_r,\ea\right.\q l \in[s,T].
\ee
Next, given $(t, s, \bx)\in \dbT^2_+\times \dbX$, consider the following standard FBSDE:
\bel{BSDEtsth}\left.\1n\ba{ll}
\ds\wt X^{s,\bx}_{r,l }:=\bx_l +\int_s^rb(l ,r',X_\cd^{s,\bx})dr'+\int_s^r\si
(l ,r',X_\cd^{s,\bx})dW_{r'},\q s\les r\les l \les T;\\
\ns\ds\wt Y^{t,s,\bx}_l =g(t,X_\cd^{s,\bx})+\int_l ^Tf(t,r,X_\cd^{s,\bx},Y^{s,\bx}_r,
\wt Z^{t,s,\bx}_r)dr-\int_l ^T\wt Z^{t,s,\bx}_rdW_r,\q l \in[s,T].\ea\right.\ee
We now define
\bel{Uts}U(t,s,\bx):=\wt Y^{t,s,\bx}_s,\q(t,s,\bx)\in\dbT^2_+\times\dbX.\ee
It is obvious that $\wt Y^{t,s,\bx}_l $ is $\si(W_r-W_s,r\in[s,l ])$-measurable,
so the above $U(t,s,\bx)$ is deterministic.

\begin{theorem}
\label{thm-FK2} \sl
Under {\rm Assumptions \ref{assum-FSVIE}} and {\rm\ref{assum-BSVIE}},
if the function $U$ defined by \rf{Uts} is continuous in all variables and,
for any fixed $t$, $\pa_\bx U(t,\cd), \pa^2_{\bx\bx} U(t,\cd)$ exist and satisfy the requirements in {\rm Definition \ref{hatC12}} (in the sense of {\rm Definition \ref{C12}}),
then $U$ is a classical solution of the PPDE \rf{UPPDE}.
\end{theorem}

\begin{proof}
First, note that $U(s,s,\bx) = \wt Y^{s,s,\bx}_s = Y^{s,\bx}_s$.
For any $s\les r\les l \les T$, by \rf{flowX2} we have
\bel{flowX3}
X^{s,\bx}_l =X^{r,\h X^{s,\bx,r}}_l ,\q\mbox{where}\q
\h X^{s,\bx,r}_{r'}:=\big(X^{s,\bx}\oplus_r\wt X^{s,\bx}_r\big)_{r'}:=X^{s,\bx}_{r'} {\bf1}_{[0, r)}(r')+\wt X^{s,\bx}_{r,r'}{\bf1}_{[r,T]}(r').
\ee
Then by the uniqueness of BSVIEs and BSDEs we have
\bel{flowY3}
Y^{s,\bx}_r=U(r,r,\h X^{s,\bx,r}),\q\wt Y^{t,s,\bx}_r=U(t,r,\h X^{s,\bx,r}).
\ee

\ss

We next establish the representation for $\wt Z^{t,s,\bx}_r$:
\bel{ZpaxU}
\wt Z^{t,s,\bx}_r = \lan \pa_\bx U(t, r, \h X^{s,\bx, r}), \si^{r, \h X^{s,\bx, r}}\ran,  \q\as, \q s\les r\les T.
\ee
Fix $\d>0$ and let $s=s_0<\cds<s_n=T$ be  such that $\D s_i := s_i-s_{i-1}\les \d$.
Denote
$$
Y^n_r := U(t, s_{i+1}, \h X^{s,\bx, r}),\q Z^n_r:= \lan \pa_\bx U(t, s_{i+1}, \h X^{s,\bx, r}), \si^{r, \h X^{s,\bx, r}}\ran,\q  r\in [s_i, s_{i+1}].
$$
Note that $(Y^n_r,Z^n_r)$ is $\cF_r$-measurable.
Fix $t$ and then apply the functional It\^{o} formula \rf{Ito-formula} to $U(t,s_{i+1},\h X^{s,\bx, \cd})$ (with time variable fixed),
we get
$$
d Y^n_r= \Big[{1\over2}\big\lan \pa^2_{\bx\bx}U(t,s_{i+1},\h X^{s,\bx, r}),\,(\si^{r,\h X^{s,\bx, r}},\si^{r,\h X^{s,\bx, r}})\big\ran
+\big\lan\pa_\bx U(t,s_{i+1},\h X^{s,\bx, r}),b^{r,\h X^{s,\bx, r}}\big\ran\Big] dr+Z^n_r dW_r.
$$
Denote $\D Y^n_r := Y^n_r - \wt Y^{t,s,\bx}_r$, $\D Z^n_r := Z^n_r - \wt Z^{t,s,\bx}_r$.
Note that $\D Y^n_{s_{i+1}}=0$, and
\bel{dYnest}
\begin{aligned}
&d \D Y^n_r =  \h f(t, r, \h X^{s,\bx,r}, Y^{s,\bx}_r, Z^n_r + \D Z^n_r) dr + \D Z^n_r dW_r,\q \mbox{where}\\
&\h f(t,r, \h\bx, y, z) :=  {1\over2}\big\lan \pa^2_{\bx\bx}U(t,s_{i+1},\h\bx),\,(\si^{r,\h\bx},\si^{r,\h\bx})\big\ran
+\big\lan\pa_\bx U(t,s_{i+1},\h \bx),b^{r,\h \bx}\big\ran + f(t, r, \h\bx, y, z).
\end{aligned}
\ee
By standard BSDE arguments we have
\bel{DZnest}
\begin{aligned}
&\dbE\Big[\sup_{s_i\les r\les s_{i+1}} |\D Y^n_r|^2+ \int_{s_i}^{s_{i+1}} |\D Z^n_r|^2 dr\Big] \\
& \q\les C\dbE\Big[ \Big(\int_{s_i}^{s_{i+1}} \big| \h f(t, r, \h X^{s,\bx,r}, Y^{s,\bx}_r, Z^n_r)\big|dr\Big)^2\Big]
\les C\big[1+\|\bx\|^{4+2\k}\big]\d \D s_{i+1},
\end{aligned}
\ee
where $\k$ is the generic order of polynomial growth in Definition \ref{hatC12}. Thus
\beaa
\dbE\Big[\sum_{i=0}^{n-1}\int_{s_i}^{s_{i+1}} \Big|\wt Z^{t,s,\bx}_r - \lan \pa_\bx U(t, s_{i+1}, \h X^{s,\bx, r}), \si^{r, \h X^{s,\bx, r}}\ran\Big|^2 dr\Big]
\les  C\big[1+\|\bx\|^{4+2\k}\big]\d.
\eeaa
Send $\d\to 0$, by the (right) continuity of $\pa_\bx U$ we obtain \rf{ZpaxU}.

\ms

Moreover, set $s_1:= s+\d$, by \rf{dYnest} again we have
\bel{DdU}
U(t,s+\d,\bx) - U(t, s, \bx) = \D Y^n_{s} =  -\dbE\Big[ \int_s^{s+\d} \bar f(t, r, X^{s,\bx}, Y^{s,\bx}_r, Z^n_r) dr\Big] - R(\d),
\ee
where
\begin{align*}
&\bar f(t,r, \h\bx, y, z) :=  {1\over2}\big\lan \pa^2_{\bx\bx}U(t,r,\h\bx),(\si^{r,\h\bx},\si^{r,\h\bx})\big\ran
+\big\lan\pa_\bx U(t,r,\h \bx),b^{r,\h \bx}\big\ran + f(t, r, \h\bx, y, z),\\
&R(\d):= \dbE\Big[ \int_s^{s+\d}\big[\h f(t, r, X^{s,\bx}, Y^{s,\bx}_r, Z^n_r + \D Z^n_r)
- \h f(t, r, X^{s,\bx}, Y^{s,\bx}_r, Z^n_r)\big] dr\\
&\qq\qq+\int_s^{s+\d}\big[\h f(t, r, X^{s,\bx}, Y^{s,\bx}_r, Z^n_r)
- \bar f(t, r, X^{s,\bx}, Y^{s,\bx}_r, Z^n_r)\big] dr\Big].
\end{align*}
By \rf{DZnest} and the regularity of $U$ we have
\beaa
|R(\d)|^2 \les    C\d \dbE\Big[ \int_s^{s+\d}|\D Z^n_r|^2 dr\Big]+o(\d^2) \les  C\big[1+\|\bx\|^{4+2\k}\big]\d^3+o(\d^2).
\eeaa
Divide the both sides of \rf{DdU} by $\d$ and send $\d\to 0$,
by the desired continuity we see that $\pa_s U(t,s,\bx)$ exists and
\begin{align*}
&-\pa_s U(t,s,\bx)= \bar f\big(t, s, \bx, U(s,s,\bx),  \lan \pa_\bx U(t,s,\bx), \si^{s,\bx}\ran\big)\\
&\q= {1\over2}\lan \pa^2_{\bx\bx}U(t,s,\bx),(\si^{s,\bx},\si^{s,\bx})\ran+\lan\pa_\bx U(t,s,\bx),b^{s,\bx}\ran
+ f(t, s, \bx, U(s,s,\bx),  \lan \pa_\bx U(t,s,\bx), \si^{s,\bx}\ran).
\end{align*}
This implies that $\pa_s U(t,s,\bx)$ has the desired regularity and $\cL U(t,s,\bx) =0$.

\ss
Finally, clearly $U(t, T, \bx) =\wt Y^{t, T, \bx}_T = g(t, T, \bx)$, thus $U$ is a classical solution of PPDE \rf{UPPDE}.
\end{proof}

\subsection{Classical solutions of the PPDE }

In this subsection, we provide some sufficient conditions so that the function $U$ defined by \rf{Uts} has the desired regularity
and thus is the unique classical solution of the PPDE \rf{UPPDE}.

We first note that, since the derivatives of $U$ involve c\`adl\`ag paths,
we shall assume the coefficients $b, \si, f, g$ can be extended to $\h\dbX$,
and we will use the same notations.
The derivatives of $f$ with respect to $(y, z)$
and those of $(b,\si, f, g)$ with respect to the first time variable $t$ are in the standard sense, while those with respect to the second time variable will not be needed.
Given an adapted function $\f: \hat\L\to \dbR$,
the derivative with respect to $\bx$ is the Fr\'{e}chet derivative as a linear operator on $\h\dbX$:
\bel{Frechet}
\f(t, \bx+\eta)-  \f(t, \bx) = \lan D \f(t,\bx), \eta\ran + o(\|\eta\|).
\ee
We emphasize that in \rf{u_o} and \rf{u_oo} the perturbation path $\eta$ is on $[t, T]$,
while here $\eta$ is  on $[0, T]$ (actually on $[0, t]$  due to the adaptedness).
Similarly we define $D^2 \f$ as a bilinear operator on $\h\dbX\times \h\dbX$.

We say $D\f$ is bounded if $|\lan D \f(t,\bx),\eta\ran|\les C\|\eta\|$ for all $(t,\bx, \eta)\in \h\L\times\h\dbX$,
and $D\f$ is continuous if, for any $\eta\in \h\dbX$, the mapping $(t,\bx)\in \h\L \mapsto \lan D \f(t,\bx), \eta\ran$ is continuous.
Similarly $D^2\f$ is bounded if $|\lan D^2 \f(t,\bx),
(\eta',\eta)\ran|\les C\|\eta'\|\|\eta\|$ and continuous if $(t,\bx)\in \h\L \mapsto \lan D^2 \f(t,\bx), (\eta',\eta)\ran$ is continuous.
When the mapping $\bx\in \h\dbX\mapsto D\f(t,\bx)$ is continuous, one can easily show that
\bea
\label{Frechet2}
\f(t, \bx+\eta)-  \f(t, \bx) = \int_0^1 \lan D\f(t,\bx+\th \eta), \eta\ran d\th.
\eea
Moreover, we may switch the order of differentiation: $\pa_t D\f = D\pa_t \f$, if one of them is continuous.

\begin{assumption}
\label{assum-classical}\rm
Assumptions \ref{assum-FSVIE} and \ref{assum-BSVIE} hold, and the dependence of $b,\si, f, g$ on $\bx$ can be extended to $\h\L$, stilled using the same notations, such that
\begin{enumerate}[(i)]

\item  For $\f= b, \si, g$, $\f$ is twice differentiable in $\bx$ with bounded derivatives, and $D^2\f$ is uniformly Lipschitz continuous in $\bx$;

\item $f$ is jointly twice differentiable in $(\bx, y, z)$  with bounded derivatives, and the second order derivatives are uniformly Lipschitz continuous in $(\bx, y, z)$;

\item $b, \si$ are differentiable in the first time variable $t$, and $\pa_t b, \pa_t \si$ satisfy the requirements in (i).
\end{enumerate}
\end{assumption}

\begin{theorem}\label{thm-classical} \sl
Under {\rm Assumption \ref{assum-classical}}, the function  $U$ defined by \rf{Uts}  is the unique classical solution of  PPDE \rf{UPPDE}.
\end{theorem}

\begin{proof}
By Theorem \ref{thm-FK2}, it suffices to verify the required regularities of $U$.
We shall repeatedly apply Propositions \ref{prop-FSVIE2} and \ref{prop-BSVIE2}.
In the proof we may abuse the notations $\D X$, $\check b$ etc, and we may omit the variable $\omega$.
We proceed in three steps.

\ss

{\it Step 1.} In this step we  show that
\bel{paxU}
\lan \pa_\bx U(t,s, \bx), \eta\ran = \nabla_\eta \wt Y^{t, s,\bx}_s,\q (s,\bx) \in \h \L,~ t\les s,~\eta \in \dbX_s,
\ee
where $(\nabla_\eta X^{s,\bx}, \nabla_\eta Y^{s,\bx}, \nabla_\eta\wt Y^{t,s,\bx})$ solve the following linear system with random coefficients on $[s,T]$: denoting $\nabla_\eta X^{s,\bx}_l :=0$ for $l \in[0,s]$,
\bel{nablaX}\ba{ll}
\ns\ds\nabla_\eta X^{s,\bx}_l =\eta_l +\int_s^l \lan Db(l ,r,X^{s,\bx}),\nabla_\eta X^{s,\bx}\ran dr+\int_s^l \lan D\si(l ,r,X^{s,\bx}),\nabla_\eta X^{s,\bx}\ran dW_r;\\
\ns\ds\nabla_\eta Y^{s,\bx}_l =\lan Dg(l ,X^{s,\bx}),\nabla_\eta X^{s,\bx}\ran
-\int_l ^T\nabla_\eta Z^{l ,s,\bx}_rdW_r\\
\ns\ds\qq+\int_l ^T\Big[\lan Df(\cd),\nabla_\eta X^{s,\bx}\ran+\pa_yf(\cd)\nabla_\eta Y^{s,\bx}_r+\pa_zf(\cd)\nabla_\eta Z^{l ,s,\bx}_r\Big](l ,r,X^{s,\bx},Y^{s,\bx}_r,
Z^{l ,s,\bx}_r)dr;\\
\ns\ds\nabla_\eta\wt Y^{t,s,\bx}_l =\lan Dg(t,X^{s,\bx}),\nabla_\eta X^{s,\bx}\ran
-\int_l ^T\nabla_\eta\wt Z^{t,s,\bx}_r dW_r\\
\ns\ds\qq+\int_l ^T\Big[\lan Df(\cd),\nabla_\eta X^{s,\bx}\ran+\pa_yf(\cd)
\nabla_\eta Y^{s,\bx}_r+\pa_zf(\cd)\nabla_\eta\wt Z^{t,s,\bx}_r\Big]
(t,r,X^{s,\bx},Y^{s,\bx}_r,\wt Z^{t, s,\bx}_r)dr.
\ea\ee

Indeed, first by Propositions \ref{prop-FSVIE2} and \ref{prop-BSVIE2}, and by standard BSDE arguments (see, e.g., \cite[Chapter 4]{Zhang-2017}) we see that the above system \rf{nablaX} is wellposed, and
\bel{nablaXest}\ba{ll}
\ns\ds\dbE\big[\|\nabla_\eta X^{s,\bx}\|^p\big]\les C_p\|\eta\|^p; \\
\ns\ds\sup_{l \in[s,T]}\dbE\Big[ |\nabla_\eta Y^{s,\bx}_l |^p+\Big(\int_l ^T|\nabla_\eta Z^{l , s,\bx}_r|^2dr\Big)^{p\over 2}\Big]\les C_p\dbE\big[\|\nabla_\eta X^{s,\bx}\|^p\big] \les C_p \|\eta\|^p;\\
\ns\ds\dbE\Big[\sup_{l \in [s,T]}|\nabla_\eta \wt Y^{t,s,\bx}_l |^p +\Big(\int_s^T |\nabla_\eta \wt Z^{t,s,\bx}_r|^2dr\Big)^{p\over 2}\Big]\\
\ns\ds\q\les C_p\dbE\big[\|\nabla_\eta X^{s,\bx}\|^p\big]+C_p\sup_{l \in [s,T]}\dbE\big[ |\nabla_\eta Y^{s,\bx}_l |^p\big]\les C_p \|\eta\|^p.
\ea\ee
Next, denote
$$
\D_\eta X^{s,\bx}:= X^{s,\bx+\eta}-X^{s,\bx}-\nabla_\eta X^{s,\bx},\q\D_\eta\wt Y^{t,s, \bx}:=\wt Y^{t,s,\bx+\eta}-\wt Y^{t,s,\bx}-\nabla_\eta\wt Y^{t,s,\bx},
$$
and similarly for $\D_\eta Y^{s,\bx},\, \D_\eta Z^{l ,s,\bx},\, \D_\eta\wt Z^{t,s,\bx}$.
Then $\D_\eta X^{s,\bx}_l =0$ for $l \in[0,s]$, and for $l \in[s,T]$,
\bel{DetaX}\ba{ll}
\ns\ds\D_\eta X^{s,\bx}_l =\int_s^l \check b(l ,r,\D_\eta X^{s,\bx})dr+\int_s^l \check \si(l ,r,\D_\eta X^{s,\bx})dW_r;\\
\ns\ds\D_\eta Y^{s,\bx}_l =\check g(l ,\D_\eta X^{s,\bx})
+\int_l ^T\check f(l ,r,\D_\eta X^{s,\bx},\D_\eta Y^{s,\bx}_r,\D_\eta Z^{l ,s,\bx}_r)dr- \int_l ^T\D_\eta Z^{l ,s,\bx}_rdW_r;\\
\ns\ds\D_\eta\wt Y^{t,s,\bx}_l =\check g(t,\D_\eta X^{s,\bx})
+\!\! \int_l ^T\check f(t,r,\D_\eta X^{s,\bx},\D_\eta Y^{s,\bx}_r,
\D_\eta\wt Z^{t,s,\bx}_r)dr-\!\!\int_l ^T\D_\eta\wt Z^{t,s,\bx}_r dW_r,\ea\ee
where, for $\f=b,\si,g$,
\beaa
\ba{ll}
\check\f(l ,r,\bx'):=\f(l ,r,X^{s,\bx}+\nabla_\eta X^{s,\bx}+\bx')-\f(l ,r, X^{s,\bx})-\lan D\f(l ,r,X^{s,\bx}),\nabla_\eta X^{s,\bx}\ran;\\
\ns\ds\check f(l ,r,\bx',y,z):=f(l ,r,X^{s,\bx}+\nabla_\eta X^{s,\bx}+\bx',Y^{s,\bx}_r + \nabla_\eta Y^{s,\bx}_r+y,Z^{l ,s,\bx}_r+\nabla_\eta Z^{l ,s,\bx}_r+z)\\
\ns\ds\qq-\Big[f(\cd)+\lan Df(\cd),\nabla_\eta X^{s,\bx}\ran+\pa_yf(\cd)\nabla_\eta Y^{s,\bx}_r+\pa_zf(\cd)\nabla_\eta Z^{l ,s,\bx}_r\Big](l ,r,X^{s,\bx},Y^{s,\bx}_r,Z^{l , s,\bx}_r).
\ea
\eeaa
Again by Propositions \ref{prop-FSVIE2} and \ref{prop-BSVIE2},  recalling \rf{Frechet2} we have
\beaa
\ba{ll}
\ns\ds\dbE\big[\|\D_\eta X^{s,\bx}_l \|^p\big] \les C_p\sup_{(l ,r)}\sum_{\f=b,\si}\dbE\Big[|\check\f(l ,r,\b0)|^p+|\pa_t\check\f(l ,r, \b0)|^p\Big]\\
\ns\ds\qq\les C_p\sup_{(l ,r)}\sum_{\f=b,\si}\dbE\bigg[\Big|\int_0^1 \big\lan D\f(l ,r, X^{s,\bx}+\th \nabla_\eta X^{s,\bx})-D\f(l ,r,X^{s,\bx}),\nabla_\eta X^{s,\bx}\big\ran d\th \Big|^p\\
\ns\ds\qq\q+\Big|\int_0^1 \pa_t \big[\big\lan D\f(l ,r,X^{s,\bx}+\th \nabla_\eta X^{s,\bx})
-D\f(l ,r,X^{s,\bx}),\nabla_\eta X^{s,\bx}\big\ran\big]d\th\Big|^p\bigg] \\
\ns\ds\qq\les C_p\dbE\big[\|\nabla_\eta X^{s,\bx}\|^{2p}\big]\les C_p\|\eta\|^{2p},
\ea
\eeaa
and
\beaa
\ba{ll}
\ns\ns\ds\sup_{l \in[s,T]}\dbE\Big[|\D_\eta Y^{s,\bx}_l |^p+\Big(\int_l ^T|\D_\eta Z^{l , s,\bx}_r|^2dr\Big)^{p\over2}\Big] \\
\ns\ds\qq\les C_p\sup_{l \in[s,T]}\dbE\Big[|\check g(l ,\D_\eta X^{s,\bx})|^p+ \Big(\int_l ^T|\check f(l ,r,\D_\eta X^{s,\bx},0,0)|dr\Big)^p\Big]\\
\ns\ds\qq\les C_p\sup_{l \in[s,T]}\dbE\Big[|\check g(l ,\b0)|^p+\Big(\int_l ^T|\check f(l , r,\b0,0,0)|dr\Big)^p+\|\D_\eta X^{s,\bx}\|^p\Big]\\
\ns\ds\qq\les C_p\sup_{l \in[s,T]}\dbE\Big[\(\int_l ^T\big(\|\nabla_\eta X^{s,\bx}\|+ |\nabla_\eta Y^{s,\bx}_r|+|\nabla_\eta Z^{l ,s,\bx}_r|\big)^2dr\)^p
+\|\D_\eta X^{s,\bx}\|^p\Big]\\
\ns\ds\qq\les C_p\|\eta\|^{2p}.\ea
\eeaa
Then it follows from standard BSDE arguments that
\beaa
\ba{ll}
\ns\ds\dbE\Big[\sup_{l \in[s,T]}|\D_\eta\wt Y^{t,s,\bx}_l |^p+\Big(\int_s^T|\D_\eta\wt Z^{t,s,\bx}_r|^2dr\Big)^{p\over 2}\Big]\\
\ns\ds\qq\les C_p\dbE\Big[|\check g(t,\D_\eta X^{s,\bx})|^p
+\Big(\int_s^T |\check f(t,r,\D_\eta X^{s,\bx},\D_\eta Y^{s,\bx}_r,0)|dr\Big)^p\Big]\\
\ns\ds\qq\les C_p\dbE\Big[|\check g(t,\D_\eta X^{s,\bx})|^p
+\Big(\int_s^T\big(|\check f(t,r,\D_\eta X^{s,\bx},0,0)|+|\D_\eta Y^{s,\bx}_r|\big)dr\Big)^p\Big]\\
\ns\ds\qq\les C_p\dbE\big[\|\D_\eta X^{s,\bx}_l \|^p\big]+C_p\sup_{l \in[s,T]}\dbE\big[|\D_\eta Y^{s,\bx}_l |^p\big]\les C_p\|\eta\|^{2p}.\ea
\eeaa
In particular, taking $l=s$, this implies
$$|U(t,s,\bx+\eta)-U(t,s,\bx)-\nabla_\eta\wt Y^{t,s,\bx}_s|=|\D_\eta\wt Y^{t,s,\bx}_s|\les C\|\eta\|^2,$$
which exactly means \rf{paxU}.

\ss

{\it Step 2.} Denote
$$\ba{ll}
\ns\ds G(l ,\bx'):= \lan Dg(l ,X^{s,\bx}),\bx'\ran+\lan D^2g(l ,X^{s,\bx}),(\nabla_{\eta'} X^{s,\bx},\nabla_\eta X^{s,\bx})\ran;\\
\ns\ds F(l ,r,\bx',y,z):=\Big[\lan Df(\cd),\bx'\ran+\pa_yf(\cd)y+\pa_zf(\cd)z\\
\ns\ds\qq+\lan D^2f(\cd),(\nabla_{\eta'}X^{s,\bx},\nabla_\eta X^{s,\bx})\ran+\pa_y\lan D f(\cd),\nabla_\eta X^{s,\bx}\ran\nabla_{\eta'}Y^{s,\bx}_r\\
\ns\ds\qq+\pa_z\lan Df(\cd),\nabla_\eta X^{s,\bx}\ran\nabla_{\eta'}Z^{l ,s,\bx}_r+\lan D \pa_yf(\cd),\nabla_{\eta'}X^{s,\bx}\ran\nabla_\eta Y^{s,\bx}_r+\pa^2_yf(\cd)\nabla_\eta Y^{s,\bx}_r\nabla_{\eta'}Y^{s,\bx}_r \\
\ns\ds\qq+\pa^2_{yz}f(\cd)\nabla_\eta Y^{s,\bx}_r\nabla_{\eta'}Z^{l ,s,\bx}_r+\lan D\pa_z f(\cd),\nabla_{\eta'}X^{s,\bx}\ran\nabla_\eta Z^{l ,s,\bx}_r+\pa^2_{yz}f(\cd) \nabla_{\eta'}Y^{s,\bx}_r\nabla_\eta Z^{l ,s,\bx}_r\\
\ns\ds\qq+\pa^2_{zz}f(\cd)\nabla_{\eta'}Z^{l ,s,\bx}_r\nabla_\eta Z^{l ,s,\bx}_r\Big](l , r,X^{s,\bx},Y^{s,\bx}_r,Z^{l ,s,\bx}_r).\ea$$
Following similar arguments as in Step 1, we can show that
\bel{paxxU}
\lan\pa^2_{\bx\bx}U(t,s,\bx),(\eta',\eta)\ran=\nabla_{\eta',\eta}\wt Y^{t, s,\bx}_s,\q(s,\bx)\in\h\L,~ t\les s,~\eta',\eta\in\dbX_s,
\ee
with $\nabla_{\eta',\eta}\wt Y^{t,s,\bx}$ solving the following linear system on $[s,T]$:
denoting $\nabla_{\eta',\eta}X^{s,\bx}_l :=0$ for $l \in[0,s]$,
$$\ba{ll}
\ns\ds\nabla_{\eta',\eta}X^{s,\bx}_l =\int_s^l \Big[\big\lan Db(l ,r,X^{s,\bx}), \nabla_{\eta',\eta}X^{s,\bx}\big\ran+\big\lan D^2b(l ,r,X^{s,\bx}),(\nabla_{\eta'}X^{s,\bx}, \nabla_{\eta}X^{s,\bx})\big\ran\Big]dr\\
\ns\ds\qq+\int_s^l \Big[\big\lan D\si(l ,r,X^{s,\bx}),\nabla_{\eta',\eta}X^{s,\bx}\big\ran+\big\lan
D^2\si(l ,r,X^{s,\bx}),(\nabla_{\eta'}X^{s,\bx},\nabla_\eta X^{s,\bx})\big\ran\Big]dW_r;\\
\ns\ds\nabla_{\eta',\eta}Y^{s,\bx}_l =G(l ,\nabla_{\eta',\eta}X^{s,\bx})-\int_l ^T \nabla_{\eta',\eta}Z^{l ,s,\bx}_rdW_r \\
\ns\ds\qq\qq\q+\int_l ^TF(l ,r,\nabla_{\eta',\eta}X^{s,\bx},\nabla_{\eta',\eta} Y^{s,\bx}_r,\nabla_{\eta',\eta}Z^{l ,s,\bx}_r)dr;\\
\ns\ds\nabla_{\eta',\eta}\wt Y^{t,s,\bx}_l =G(l ,\nabla_{\eta',\eta}X^{s,\bx})-\int_l ^T \nabla_{\eta',\eta}\wt Z^{t,s,\bx}_r dW_r\\
\ns\ds\qq\qq\q+\int_l ^TF(t,r,\nabla_{\eta',\eta}X^{s,\bx},\nabla_{\eta',\eta}Y^{s,\bx}_r, \nabla_{\eta',\eta}\wt Z^{t,s,\bx}_r)dr.
\ea
$$

{\it Step 3.} It remains to show that $U, \pa_\bx U(t,\cd)$ and $\pa^2_{\bx\bx} U(t, \cd)$ have the desired regularity required in Theorem \ref{thm-FK2}. We emphasize that these functions here are already defined in $\h\dbX$.

\ss

{\it Step 3.1.} We first show the continuity in $\bx$.  Fix $(s,\bx)\in\h\L$, $t\les s$, and $\bx'\in\h\dbX$.
By abusing the notations, denote $\D_{\bx'} X^{s, \bx}:= X^{s, \bx + \bx'}-X^{s,\bx}$ and similarly for the other terms, and
$$\ba{ll}
\ns\ds\check\f(l ,r,\bx'):=\f(l ,r,X^{s,\bx}+\bx')-\f(l ,r,X^{s,\bx}),\q\hbox{for }\f=b,\si,g;\\
\ns\ds\check f(l ,r, \bx', y,z):=f(l ,r,X^{s,\bx+\bx'},Y^{s,\bx}_r+y,Z^{l ,s,\bx}_r+z)-f(l ,r, X^{s,\bx},Y^{s,\bx}_r,Z^{l ,s,\bx}_r).
\ea$$
We can see that
$$\ba{ll}
\ns\ds\D_{\bx'} X^{s,\bx}_l =\bx'_l ,\qq l \in[0,s];\\
\ns\ds\D_{\bx'} X^{s,\bx}_l =\bx'_l +\int_s^l \check b(l ,r,\D_{\bx'} X^{s,\bx})dr
+\int_s^l \check\si(l ,r,\D_{\bx'} X^{s,\bx})dW_r,\qq l \in[s,T];
\ea$$
and $\D_{\bx'} Y^{s,\bx}, \D_{\bx'} \wt Y^{t, s,\bx}$ satisfy equations similar to the last two equations in \rf{DetaX}.
Following the same arguments as in  Step 1 we can easily show that
\bel{Ubx}
|U(t,s,\bx+\bx')-U(t,s,\bx)|\les C\|\bx'\|,\qq\forall\bx'\in \h\dbX.
\ee

Similarly, for any fixed $\eta,\eta'\in\dbX_s$ with $\|\eta\|,\|\eta'\|\les1$,
one can show that $\nabla_\eta\wt Y^{t,x,\bx}_s$ and $\nabla_{\eta',\eta}\wt Y^{t,x,\bx}_s$ are uniformly Lipschitz continuous in $\bx$;
 that is, for any $\bx'\in \h\dbX$,
\bel{Ubx2}
\big|\big\lan \pa_\bx U(t,s,\bx+\bx')-  \pa_\bx U(t,s,\bx),\, \eta\big\ran \big|
+ \big|\big\lan \pa^2_{\bx\bx} U(t,s,\bx+\bx')-  \pa^2_{\bx\bx} U(t,s,\bx),\, (\eta',\eta)\big\ran\big|\les C\|\bx'\|.
\ee

{\it Step 3.2.} We next show the right continuity in $s$. Recall \rf{flowX3} and \rf{flowY3}, we have
$$
\wt Y^{t,s,\bx}_l =U(t,s+\d,\h X^{s,\bx,s+\d})+\int_l ^{s+\d}f(t,r,X^{s,\bx}, Y^{s,\bx}_r,\wt Z^{t,s,\bx}_r)dr-\int_l ^{s+\d}\wt Z^{t,s,\bx}_rdW_r,
$$
for $l\in [s, s+\d]$. Then by Propositions \ref{prop-FSVIE2} and \ref{prop-BSVIE2}, we get
\bel{Uest}\ba{ll}
\ns\ds|U(t,s,\bx)-U(t,s+\d,\bx)|^2=\dbE\Big[\big|\dbE_s\big[\wt Y^{t,s,\bx}_s-U(t,s+\d,\bx)\big]\big|^2\Big] \\
\ns\ds\q\les C\dbE\Big[\big|U(t,s+\d,\h X^{s,\bx,s+\d})-U(t,s+\d,\bx)\big|^2+\Big(\int_s^{s+\d} |f(t,r,X^{s,\bx},Y^{s,\bx}_r,\wt Z^{t,s,\bx}_r)|dr\Big)^2\Big]\\
\ns\ds\q\les C\dbE\Big[\|\h X^{s,\bx,s+\d}-\bx\|^2+ \d \int_s^{s+\d}\big(1+\|X^{s,\bx}\|^2 +|Y^{s,\bx}_r|^2+|\wt Z^{t,s,\bx}_r |^2\big)dr\Big]\\
\ns\ds\q\les C(1+\|\bx\|^2)\d.
\ea\ee
Thus
\bel{Us}
|U(t,s+\d,\bx)-U(t,s,\bx)|\les C(1+\|\bx\|)\sqrt\d.
\ee
Similarly, fix $\eta\in \dbX_s$, by \rf{BSDEtsth}, \rf{flowX3} and \rf{flowY3} again, we have
\begin{align*}
\nabla_\eta\wt Y^{t,s,\bx}_{s+\d}&=\lim_{\e\to0}{1\over\e}\Big[\wt Y^{t,s,\bx+\e \eta}_{s+\d}-\wt Y^{t,s,\bx}_{s+\d}\Big]
=\lim_{\e\to0}{1\over\e}\Big[U(t,s+\d,\h X^{s,\bx+\e\eta,s+\d})-U(t,s+\d,\h X^{s,\bx, s+\d})\Big]\\
&=\lan\pa_\bx U(t,s+\d,\h X^{s,\bx,s+\d}),\nabla_\eta \h X^{s,\bx, s+\d}\ran,
\end{align*}
where
$$\ba{ll}
\ns\ds\nabla_\eta\h X^{s,\bx, s+\d}_l :=\lim_{\e\to 0}{1\over\e}\Big[\h X^{s,\bx+\e\eta, s+\d}_l -\h X^{s,\bx, s+\d}_l \Big]=\lim_{\e\to0}{1\over\e}\Big[\wt X^{s,\bx+\e\eta}_{s+\d, l }-\wt X^{s,\bx}_{s+\d,l }\Big]\\
\ns\ds\qq=\eta_l +\int_s^{s+\d}\lan\pa_\bx b(l ,r,X^{s,\bx}),\nabla_\eta X^{s,\bx}\ran dr
+\int_s^{s+\d}\lan\pa_\bx\si(l ,r,X^{s,\bx}),\nabla_\eta X^{s,\bx}\ran dW_r,
\ea$$
with $\nabla_\eta X^{s,\bx}$ determined by \rf{nablaX}.
From the above and \rf{nablaX}, note that, for $l \in[s, s+\d]$,
$$\ba{ll}
\ns\ds\nabla_\eta\wt Y^{t,s,\bx}_l =\lan\pa_\bx U(t,s+\d,\h X^{s,\bx,s+\d}),\nabla_\eta
\h X^{s,\bx,s+\d}\ran-\int_l ^{s+\d}\nabla_\eta\wt Z^{t,s,\bx}_rdW_r\\
\ns\ds\q+\int_\t^{s+\d}\Big[\lan Df(\cd),\nabla_\eta X^{s,\bx}\ran+\pa_yf(\cd)\nabla_\eta Y^{s,\bx}_r+\pa_zf(\cd)\nabla_\eta\wt Z^{t,s,\bx}_r\Big](t,r,X^{s,\bx},Y^{s,\bx}_r,\wt Z^{t, s,\bx}_r)dr.
\ea$$
Then similar to \rf{Uest}  we have
$$\ba{ll}
\ns\ds\big|\lan\pa_\bx U(t,s,\bx),\eta\ran-\lan\pa_\bx U(t,s+\d,\bx),\eta\ran\big|^2
=\big|\nabla_\eta\wt Y^{t,s,\bx}_s-\lan\pa_\bx U(t,s+\d,\bx),\eta\ran\big|^2\\
\ns\ds\q\les C\dbE\Big[\big|\lan\pa_\bx U(t,s+\d,\h X^{s,\bx,s+\d}),\nabla_\eta\h X^{s,\bx, s+\d}\ran-\lan\pa_\bx U(t,s+\d,\bx),\eta\ran\big|^2\\
\ns\ds\q\qq+\Big(\int_s^{s+\d}\big|\lan Df(\cd),\nabla_\eta X^{s,\bx}\ran+\pa_yf(\cd) \nabla_\eta Y^{s,\bx}_r+\pa_zf(\cd)\nabla_\eta\wt Z^{t,s,\bx}_r\big|dr\Big)^2\Big]\\
\ns\ds\q\les C\dbE\Big[\|\h X^{s,\bx,s+\d}-\bx\|^2\|\nabla_\eta\h X^{s,\bx,s+\d}\|^2+\| \nabla_\eta\h X^{s,\bx,s+\d}-\eta\|^2\\
\ns\ds\q\qq\times\d\int_s^{s+\d}\big(\|\nabla_\eta X^{s,\bx}\|^2+|\nabla_\eta Y^{s,\bx}_r|^2+|\nabla_\eta \wt Z^{t,s,\bx}_r|^2\big)dr\Big]\\
\ns\ds\q\les C(1+\|\bx\|^2)\|\eta\|^2\d.
\ea$$
Thus
\beaa
\big|\lan\pa_\bx U(t,s+\d,\bx),\eta\ran-\lan\pa_\bx U(t,s,\bx),\eta\ran\big|\les C(1+\|\bx\|)\|\eta\|\sqrt\d.
\eeaa
Similarly, by using \rf{paxxU} we can show that
\beaa
\big|\lan\pa^2_{\bx\bx}U(t,s+\d,\bx),\,(\eta',\eta)\ran-\lan\pa^2_{\bx\bx} U(t,s,\bx),\,(\eta',\eta)\ran\big|\les C(1+\|\bx\|)\|\eta'\|\|\eta\|\sqrt\d.
\eeaa

{\it Step 3.3.} Finally, by \rf{Holderf} and standard BSDE arguments we have
\beaa
|U(t-\d,s,\bx)-U(t,s,\bx)|=|\wt Y^{t-\d,s,\bx}_s-\wt Y^{t,s,\bx}_s|\les C(1+\|\bx\|) \rho(\d).
\eeaa
This, together with \rf{Ubx} and \rf{Us}, implies that $U$ is continuous in all variables $(t,s,\bx)$.
\end{proof}

Combining Theorems \ref{thm-FK1} and \ref{thm-classical}, under  Assumption \ref{assum-classical},
the {\it path dependent Feynman--Kac formula} of FBSVIEs is established in the contexts of classical solutions. In Subsection \ref{sect-paxU} below, we shall obtain a more explicit representation formula for $\pa_\bx U(t,s,\bx)$.

\section{Viscosity solution of the PPDE}
\label{sect-viscosity}
\setcounter{equation}{0}
Inspired by Proposition \ref{prop-comparison}, in this section we investigate viscosity solutions for the PPDE system \rf{UPPDE} in the case $m=1$.
Since the state space $\dbX$ is not locally compact here, we shall take the approach of Ekren--Keller--Touzi--Zhang \cite{EKTZ},
rather than the standard approach of Crandall--Ishii--Lions  \cite{GIL-1992}.
However, we shall emphasize that the paths here are on the whole interval $[0, T]$,
due to the Volterra nature of the state process, which is different from the setting in \cite{EKTZ}.
In particular, our work covers the PPDE in Viens--Zhang \cite{Viens-Zhang-2019} (under some stronger technical conditions though).

\ss

Throughout this section, we shall assume the following.

\begin{assumption}
\label{assum-viscosity}\rm
Let Assumptions \ref{assum-FSVIE} and \ref{assum-BSVIE} hold and $m=1$. Moreover,
\begin{enumerate}[(i)]

\item{}  $f, g$ are bounded and uniformly Lipschitz continuous in $\bx$.

\item{} $f$ is nondecreasing in $y$.
\end{enumerate}
\end{assumption}

We remark that the monotonicity condition in Assumption \ref{assum-viscosity} (ii) is essentially the {\it proper} condition in \cite{GIL-1992} for elliptic PDEs. For standard parabolic PDE like \rf{PDE1}, such a condition is redundant because, for any Lipschitz continuous function $f$, $\wt u(t,x):= e^{-\l t} u(t,x)$ will satisfy a PDE whose corresponding $\wt f$ is nondecreasing in $y$ whenever $\l$ is large enough. However, due to the two time variable structure, this change variable technique does not work for PPDE \rf{UPPDE}. Indeed, if we remove the monotonicity condition, the comparison principle may fail even for classical solutions.

Let $C^0_b( \dbT^2_+\times \dbX)$ denote the set of functions $U:  \dbT^2_+\times \dbX\to \dbR$ such that $U$ is bounded, uniformly continuous in all variables, and progressively measurable. Following the arguments in the proof of Theorem \ref{thm-classical}, Step 3, we have the following.

\begin{lemma}\label{lem-cont} \sl
Under {\rm Assumption \ref{assum-viscosity}}, the function $U$ defined  by \rf{Uts} is in $C^0_b( \dbT^2_+\times \dbX)$.
\end{lemma}

To introduce our notion of viscosity solutions, for any $U\in C^0_b( \dbT^2_+\times \dbX)$ and $\phi\in C^{1,2}_+(\L)$, define
\bel{LUphi}
\begin{aligned}
\cL_U\phi(t, s,\bx)
&:=[\cL_U\phi](t, s,\bx):= \pa_s \phi(s, \bx) + {1\over 2} \langle \pa^2_{\bx\bx} \phi(s,\bx), (\si^{s, \bx}, \si^{s,\bx})\rangle
+ \langle \pa_\bx \phi(s,\bx), b^{s,\bx}\rangle\\
&\q~ + f\big(t, s, \bx, U(s,s, \bx), \langle \pa_\bx \phi(s,\bx), \si^{s,\bx}\rangle\big).
\end{aligned}
\ee
We emphasize that we use $U(s,s,\bx)$ instead of $\phi(s, \bx)$ inside $f$.
It is clear that, for any fixed $t$,
\bea
\label{LUt}
\cL_U U_t(t,s,\bx) = \cL U(t,s,\bx),\q\mbox{where} \q U_t:= U(t,\cd).
\eea
For any $s\in\dbT$ and $L>0$, denote $\dbF^s:=\{\cF^t_r\}_{r\in[s,T]}$ with $\cF^s_r:= \si(W_l-W_s,l\in[s,r])$. Let $\cT_s$ be the set of $\dbF^s$-stopping times, $\cT_s^+$ the subset of $\t\in \cT_s$ such that $\t>s$, a.s., $\cU_s^L$ the set of $\dbF^s$-progressively measurable processes on $[s, T]$ bounded by $L$,
and
\bel{def-M}
M^\th_r := \exp\Big(\int_s^r \th_l dW_l - {1\over 2} \int_s^r |\th_l|^2 dl\Big),\q  r\in [s, T],~ \th\in \cU^L_s.
\ee
Given $U\in C^0_b(\dbT^2_+\times \dbX)$ and $(t,s,\bx)\in \dbT^2_+\times \dbX$, denote
\bel{cA}\left.\ba{lll}
\displaystyle \underline \cA^L U(t,s,\bx) := \Big\{\phi\in C^{1,2}_+([s, T]\times \dbX; \dbR)\bigm|\exists\,\ch\in\cT_s^+\hb{ such that}\\
\displaystyle\qq\qq\qq\qq \phi(s,\bx)-U(t,s, \bx) = 0 =  \inf_{\th\in \cU_s^L} \inf_{\ch\ges \t\in \cT_s} \dbE\big[M^\th_\t [\phi - U_t](\t, \h X^{s,\bx, \t})\big]\Big\};\\
\displaystyle \overline \cA^L U(t,s,\bx) := \Big\{\phi\in C^{1,2}_+([s, T]\times \dbX; \dbR)\bigm|\exists\,\ch\in\cT_s^+\hb{ such that}\\
\displaystyle\qq\qq\qq\qq \phi(s,\bx)-U(t,s,\bx) = 0 =  \sup_{\th\in \cU_s^L} \sup_{\ch\ges \t\in \cT_s} \dbE\big[M^\th_\t [\phi-U_t](\t,\h X^{s,\bx, \t})\big]\Big\}.
\ea\right.
\ee
We note that, if $\phi\in\underline \cA^L U(t,s,\bx)$ with the corresponding $\ch\in\cT_s^+$, then for any $\th\in \cU_{s}^L$ and $\ch\ges\t\in \cT_s$, we have
\bel{rem-A}
M^\th_s [\phi - U_t](s,\bx)= 0 \les   \dbE\big[M^\th_\t [\phi - U_t](\t, \hat X^{s,\bx, \t})\big].
\ee

\ms

\begin{definition}\label{defn-viscosity}\rm
Let $U\in C^0_b(\dbT^2_+\times \dbX)$.
\begin{enumerate}[(i)]
\item
We say $U$ is an {\it $L$-viscosity subsolution} of PPDE \rf{UPPDE} if
\bel{subsolution}
\cL_U \phi(t, s, \bx) \ges 0\q\mbox{for any $(t,s,\bx)\in \dbT_+^2\times \dbX$ and any $\phi\in \underline \cA^L U(t,s,\bx)$}.
\ee

\item
We say $U$ is an {\it $L$-viscosity supersolution} of PPDE \rf{UPPDE} if
\bel{supersolution}
\cL_U \phi(t, s, \bx) \les 0\q\mbox{for any $(t,s,\bx)\in \dbT_+^2\times \dbX$ and any $\phi\in \overline \cA^L U(t,s,\bx)$}.
\ee

\item
We say $U$ is an {\it $L$-viscosity solution} of PPDE \rf{UPPDE} if it is both an $L$-viscosity subsolution and an $L$-viscosity supersolution.
Moreover, we say $U$ is a {\it viscosity solution} of PPDE \rf{UPPDE} if it is an $L$-viscosity solution for some $L>0$.
\end{enumerate}
\end{definition}

For consistency, we say $U$ is a {\it classical subsolution} (resp. {\it classical supersolution}) of PPDE \rf{UPPDE}
if $U_t(\cd)\in C^{1,2}_+([s, T]\times \dbX; \dbR)$ and satisfies
$$
\cL_U U_t(t,s,\bx) = \cL U(t,s,\bx)\ges  ~(\mbox{resp.} \les )~ 0.
$$
From now on, we let
\bel{L0}
\mbox{$L_0$ denote the Lipschitz constant of $f$ with respect to $z$.}
\ee

We first have the consistency result.

\begin{proposition}\label{prop-consistency} \sl
Assume  $U\in C^0_b(\dbT^2_+\times \dbX)$ and $U(t\,,\cd\,,\cd) \in C^{1,2}_+([t, T]\times \dbX)$.
Then $U$ is a classical subsolution of PPDE \rf{UPPDE} if and only if it is a viscosity subsolution of PPDE \rf{UPPDE}.
\end{proposition}

\begin{proof}
We first assume $U$ is an $L$-viscosity subsolution for some $L$.
Clearly $U_t \in \underline \cA^L U(t,s,\bx)$. Then $\cL U(t,s,\bx) = \cL_U U_t(t,s,\bx) \ges 0$,
which implies that $U$ is a classical subsoution.

On the other hand, assume $U$ is a classical subsolution.
For any $\phi\in \underline \cA^{L_0} U(t,s,\bx)$ with the corresponding $\ch\in \cT_s^+$, applying the functional It\^o formula we have
\bel{dphi}\ba{ll}
\ns\ds d\phi(r,\h X^{s,\bx,r})=\Big[\pa_r\phi+{1\over2}\lan\pa^2_{\bx\bx}\phi,(\si^{\cd},  \si^{\cd})\ran+\lan\pa_\bx\phi,b^{\cd}\ran\Big](r,\h X^{s,\bx,r}) dr
+\lan\pa_\bx\phi, \si^{\cd}\ran(r,\h X^{s,\bx, r}) dW_r\\
\ns\ds\q= \big[\cL_U\phi(t,\cd)-f\big(t,\cd,U(r,\cd),\lan\pa_\bx\phi,\si^\cd\ran\big)\big](r,\h X^{s,\bx, r})dr+\lan\pa_\bx\phi,\si^\cd\ran(r,\h X^{s,\bx,r})dW_r.\ea\ee
Similarly we have
$$
dU_t(r,\h X^{s,\bx,r})=\big[\cL U(t,\cd)-f\big(t,\cd,U(r,\cd),\lan \pa_\bx U_t, \si^\cd\ran\big)\big](r,\h X^{s,\bx,r})dr+\lan\pa_\bx U_t,\si^\cd\ran(r,\h X^{s,\bx,r}) dW_r.
$$
Denote
$$
\D Y_r:=[\phi-U_t](r,\h X^{s,\bx,r}),\q\D Z_r:=\lan\pa_\bx[\phi-U_t],\si^\cd\ran(r,\h X^{s,\bx,r}).
$$
Then
$$d[\D Y_r]=\big[\cL_U\phi-\cL U\big](t,r,\h X^{s,\bx,r})dr-\th_r\D Z_r dr+\D Z_rdW_r,$$
for some $|\th|\les L_0$. This implies that
$$
d(M^\th_r\D Y_r)=M^\th_r\big[\cL_U\phi-\cL U\big](t,r,\h X^{s,\bx,r})dr+M^\th_r\D Z_r dW_r,
$$
where $M^\th_r$ is defined by \rf{def-M}. Then, for any $\t\les\ch$, by \rf{cA} (or \rf{rem-A}) we have
$$
0\les\dbE\big[M^\th_\t\D Y_\t-M^\th_s\D Y_s\big]=\dbE\[\int_s^\t M^\th_r\big[\cL_U\phi- \cL U\big](t,r,\h X^{s,\bx, r})dr\].
$$
Since $\t\ges t$ is arbitrary and $\cL_U\phi - \cL U$ is continuous, using the fact that $\cL U(t,s,\bx)\ges 0$, we have
\beaa
0\les \big[\cL_U\phi - \cL U\big](t, s, \bx) \les \cL_U\phi(t, s,\bx),
\eeaa
implying the viscosity subsolution property.
\end{proof}

Next we have the following existence of the viscosity solutions to PPDE \rf{UPPDE}.

\begin{theorem}\label{thm-existence} \sl
Under {\rm Assumption \ref{assum-viscosity}}, the function  $U$ defined  by \rf{Uts} is an $L_0$-viscosity solution of PPDE \rf{UPPDE}.
\end{theorem}

\begin{proof} Without loss of generality, we shall only verify the viscosity subsolution property. Fix $(t,s,\bx)\in \dbT^2_+\times \dbX$. Recall \rf{dphi} and \rf{BSDEtsth}, and denote
$$\D Y_r:=\phi(r,\h X^{s,\bx,r})-\wt Y^{t,s,\bx}_r,\q
\D Z_r:=\lan\pa_\bx\phi,\si^{\cd}\ran(r,\h X^{s,\bx,r})-\wt Z^{t,s,\bx}_r.$$
Then
$$\ba{ll}
\ns\ds d(\D Y_r)=\Big\{\big[\cL_U\phi(t,\cd)- f\big(t,\cd\,,U(r,\cd),\lan\pa_\bx\phi, \si^{\cd}\ran\big)\big](r,\h X^{s,\bx,r})\\
\ns\ds\qq\qq\q+f(t,r,X_\cd^{s,\bx},Y^{s,\bx}_r,\wt Z^{t, s,\bx}_r)\Big\}dr+\D Z_rdW_r.\ea$$
Recall \rf{flowX3} and \rf{flowY3}, the above implies
$$d(\D Y_r)=\cL_U\phi(t,r,\h X^{s,\bx,r})dr-\th_r\D Z_rdr+\D Z_rdW_r,$$
for some $|\th|\les L_0$.  Then, for the $M^\th_r$ defined by \rf{def-M},
$$d(M^\th_r\D Y_r)=M^\th_r\cL_U\phi(t,r,\h X^{s,\bx,r})dr+M^\th_r\D Z_rdW_r.$$
 Recall \rf{flowY3} again that $\D Y_r=[\phi-U_t](r, \h X^{s,\bx,r})$. Then, for any $\t\les\ch$,  by \rf{cA} (or \rf{rem-A}) we have
$$0\les\dbE[M^\th_\t\D Y_\t-M^\th_s\D Y_s]=\dbE\Big[\int_s^\t M^\th_r\cL_U\phi(t,r,\h X^{s,\bx,r})dr\Big].$$
Since $\t\ges t$ is arbitrary and $\cL_U\phi$ is continuous, we obtain $\cL_U\phi(t,s,\bx) \ges0$.
\end{proof}

The key for the viscosity solution theory is the following partial comparison principle.

\begin{theorem}\label{thm-partial} \sl
Let {\rm Assumption \ref{assum-viscosity}} hold and $U_1$ (resp. $U_2$) be a viscosity subsolution (resp. supersolution)  of PPDE \rf{UPPDE}.
Assume $U_1(t, T, \bx)\les U_2(t, T, \bx)$ for all $(t,\bx)\in \L$. If one of $U_1, U_2$ is smooth, then $U_1 \les U_2$.
\end{theorem}

\begin{proof}
Without loss of generality, we assume that $U_2$ is a classical supersolution.
Fix $\d>0$ which will be specified later.
We shall first prove $U_1(t,s,\bx) \les U_2(t,s,\bx)$ whenever $s\in [T-\d, T]$.
Assume by contradiction that
\bel{c}
c:= \sup_{s\in [T-\d, T],\, t\in [0, s],\, \bx\in \dbX} [U_1-U_2](t,s,\bx) >0.
\ee
Then there exists desired $(t_0,s_0,\bx^0)$ such that $[U_1-U_2](t_0,s_0,\bx^0)\ges{c\over 2}>0$. Fix $t_0$ and denote
\bel{psiV}
V(s, \bx):= [U_1-U_2](t_0, s, \bx) - {c\over 4(T-s_0)} [T-s],
\q \psi(s, \bx) := \sup_{\t\in \cT_s} \sup_{\th\in \cU^{L_0}_s} \dbE[M^\th_\t V(\t, \h X^{s, \bx, \t})],
\ee
where $\h X^{s,\bx,\t}$ and $M^\th_\t$ are defined by \rf{flowX3} and \rf{def-M}, respectively. Similar to Lemma \ref{lem-cont}, $\psi$ is bounded and uniformly continuous in $(s,\bx)$.
Moreover, by standard BSDE results (see \cite{Zhang-2017}, for example),
$\cY_s := \psi(s, \h X^{s_0, \bx^0, s})$  is the solution to the following reflected BSDE:
\bel{RBSDE}\left.\1n\ba{lll}
\ds\cY_s=V(T,\h X^{s_0,\bx^0,T})+L_0\int_s^T|\cZ_r|dr-\int_s^T\cZ_r dW_r+K_T-K_s;\\
\ns\ds\cY_s\ges V(s,\h X^{s_0,\bx^0,s}),\q\big[\cY_s-V(s,\h X^{s_0,\bx^0,s})\big]dK_s=0.
\ea\right.\ee
Denote
\bel{tau*}\tau^* := \inf\big\{s\ges s_0: \cY_s =   V(s, \h X^{s_0, \bx^0, s})\big\}.\ee
Then $dK_s = 0$ for $s\in [s_0, \t^*]$. From \rf{psiV}, we note
\begin{align}
\label{Y-s0}
& \cY_{s_0} \ges  V(s_0, \h X^{s_0, \bx^0, s_0}) =  V(s_0,  \bx^0) = [U_1-U_2](t_0,s_0, \bx^0) - {c\over 4} \ges {c\over 4}>0,\\
\label{Y-T}& \cY_T = [U_1-U_2](t_0, T, \h X^{s_0, \bx^0, T}) \les 0.
\end{align}
Then it is clear that $\dbP(\t^* < T) > 0$. Indeed, if $\dbP(\t^* < T)= 0$, we have $dK_s \equiv 0$ and the reflected BSDE \rf{RBSDE} becomes a standard BSDE.
Then the terminal condition \rf{Y-T} implies $\cY_{s_0}\les 0$, which contradicts \rf{Y-s0}.
Therefore, there exists $\omega^*\in \Omega$ such that
\beaa
\t^*(\omega^*) < T \q\mbox{and}\q \psi(s^*, \bx^*) =  V(s^*, \bx^*),\q\mbox{where}\q s^*:=  \t^*(\omega^*),\q \bx^*:= \h X^{s_0, \bx^0, s^*}(\omega^*).
\eeaa
Now define
\beaa
\phi(s, \bx):= U_2(t_0, s, \bx) + {c\over 4(T-s_0)} (T-s) + \psi(s^*, \bx^*).
\eeaa
Then $\phi\in C^{1,2}_+([s_0, T]\times \dbX)$, $\phi(s^*, \bx^*) = U_1(t_0, s^*, \bx^*)$, and, for any $\th\in \cU^{L_0}_{s^*}$ and any $\t\in \cT_{s^*}$,
\bel{phi-A}
\dbE\Big[M^\th_\t \big[\phi - U_1(t_0,\cd)\big](s, \h X^{s^*, \bx^*, s})\Big] = \psi(s^*, \bx^*) - \dbE\big[M^\th_\t  V(s, \h X^{s^*, \bx^*, s})\big] \ges 0.
\ee
That is, $\phi \in \underline \cA^{L_0}U_1(t_0, s^*, \bx^*)$, and thus by the viscosity subsolution property of $U_1$ we have
\bel{U1sub}
\begin{aligned}
0&\ \les\ \cL_{U_1} \phi(t_0, s^*, \bx^*)\\
&\ =\ \Big[ \pa_s \phi + {1\over 2} \langle \pa^2_{\bx\bx} \phi, (\si^{\cd}, \si^{\cd})\rangle
+ \langle \pa_\bx \phi, b^{\cd}\rangle + f\big(t_0, \cd, U_1(s^*,\cd), \langle \pa_\bx \phi, \si^{\cd}\rangle\big)\Big](s^*, \bx^*)\\
&\ =\ -{c\over 4(T-s_0)}+\Big[ \pa_s U_2(t_0,\cd) + {1\over 2} \langle \pa^2_{\bx\bx} U_2(t_0,\cd), (\si^{\cd}, \si^{\cd})\rangle
+ \langle \pa_\bx U_2(t_0,\cd), b^{\cd}\rangle \\
&\q\q+ f\big(t_0, \cd, U_1(s^*,\cd), \langle \pa_\bx U_2(t_0,\cd), \si^{\cd}\rangle\big)\Big](s^*, \bx^*).
\end{aligned}
\ee
Recall \rf{c} we have
\beaa
U_1(s^*, s^*, \bx^*) \les  U_2(s^*, s^*, \bx^*) +c.
\eeaa
Let $L$ denote the Lipschitz constant of $f$ with respect to $y$.
Then, by Assumption \ref{assum-viscosity} (ii) we have
   \beaa
  0&\les&  -{c\over 4\d}+\Big[ \pa_s U_2(t_0,\cd) + {1\over 2} \langle \pa^2_{\bx\bx} U_2(t_0,\cd), (\si^{\cd}, \si^{\cd})\rangle + \langle \pa_\bx U_2(t_0,\cd), b^{\cd}\rangle \\
   &&+ f\big(t_0, \cd, U_2(s^*,\cd), \langle \pa_\bx U_2(t_0,\cd), \si^{\cd}\rangle\big)\Big](s^*, \bx^*) + Lc\\
   &=&\cL U_2(t_0, s^*, \bx^*) - {c\over 4\d}  + Lc \les Lc - {c\over 4\d},
  \eeaa
  thanks to the supersolution property of $U_2$. Set $\d := {1\over 8L}$, we obtain the desired contradiction, and hence $U_1(t,s,\bx) \les U_2(t,s,\bx)$ whenever $s\in [T-\d, T]$.

\ss

Now consider the PPDE \rf{UPPDE} on $[0, T-\d]$. Since $U_1(t, T-\d, \bx) \les U_2(t, T-\d, \bx)$ for all $(t,\bx)\in [0, T-\d]\times \dbX$, by the same arguments as above we can show that $U_1(t, s, \bx) \les U_2(t, s, \bx)$ whenever $s\in [T-2\d, T-\d]$. Repeat the arguments backwardly in time, we show that $U_1 \les U_2$ over the whole space.
\end{proof}

Our final result in this section is the following comparison principle.

\begin{theorem}\label{thm-comparison} \sl
Let $b, \si$ satisfy the requirements in {\rm Assumption \ref{assum-classical}}
and $f, g$ satisfy the requirements in {\rm Assumption \ref{assum-viscosity}}.
Let $U_1$ (resp. $U_2$) be a viscosity subsolution (resp. supersolution)  of PPDE \rf{UPPDE}.
Assume $U_1(t,T,\bx)\les g(t,\bx)\les  U_2(t,T,\bx)$ for all $(t,\bx)\in\L$, then $U_1\les U_2$ on $\dbT^2_+\times \dbX$.
\end{theorem}

\begin{proof}
Without loss of generality, we shall assume $U_1(t,T,\bx) \les g(t,\bx)$ and prove only $U_1\les U$, where $U$ is defined by \rf{Uts}. We shall approximate $(f,g)$ by
$(f_n,g_n)$ which satisfy Assumption \ref{assum-classical}.
However, since $\bx$ is a path,  the standard mollification does not work and the approximations may not be uniform in terms of $\bx$. In particular, we may not have $U_1(t,T,\bx)\les g_n(t,\bx)$ for all $(t,\bx)\in\L$.
Therefore, instead of directly applying the partial comparison principle, we will follow its arguments.
As in Theorem \ref{thm-partial}, it suffices to prove $U_1(t,s,\bx) \les U(t,s, \bx)$ for $s\in [T-\d, T]$, where $\d:= {1\over 8L}$.
Assume by contradiction that, for some $s_0\in [T-\d, T]$, $t_0 \les s_0$, and $\bx^0\in \dbX$,
\bel{c2}
c:= \sup_{s\in [T-\d, T], t\in [0, s], \bx\in \dbX} [U_1-U](t,s,\bx) >0 \q\mbox{and}\q [U_1-U](t_0,s_0,\bx^0) \ges {c\over 2}.
\ee

We now mollify $(f, g)$. By first discretizing $\bx\in \dbX$,
one can easily construct $f_n, g_n$ such that, for each $n$, $f_n, g_n$ satisfy Assumption \ref{assum-classical}, and
\bea
\label{fgn}
&\displaystyle \sup_{(t, s)\in \dbT_+^2}\sup_{y,z} \big|[f_n-f](t,s, \bx, y, z)\big| + \sup_{t\in \dbT} \big|[g_n-g](t,T, \bx)\big|
\les C\Big[{1\over n} +  OSC_{1\over n}(\bx)\Big],\\
&\mbox{where}\q OSC_\d(\bx):= \sup_{|t-s|\les \d} |\bx_t - \bx_s|.\nonumber
\eea
By Theorem \ref{thm-classical}, the PPDE \rf{UPPDE} with coefficients $(b,\si, f_n, g_n)$ has a classical solution $U_n$.
As in \rf{psiV}, fix $t_0$ and denote
\bel{psiVn}
V_n(s, \bx):= [U_1-U_n](t_0, s, \bx) - {c\over 4(T-s_0)} [T-s],
\q \psi_n(s, \bx) := \sup_{\t\in \cT_s} \sup_{\th\in \cU^{L_0}_s} \dbE[M^\th_\t V_n(\t, \h X^{s, \bx, \t})].
\ee
Denote further $\cY^n_s := \psi_n(s, \h X^{s_0, \bx^0, s})$. Then \rf{RBSDE} and \rf{tau*} become:
\bel{RBSDEn}
\left.\ba{lll}
\displaystyle \cY^n_s =  V_n(T, \h X^{s_0, \bx^0, T}) + L_0 \int_s^T |\cZ^n_r| dr - \int_s^T \cZ^n_r dW_r + K^n_T - K^n_s;\\
\displaystyle \cY^n_s \ges    V_n(s, \h X^{s_0, \bx^0, s}),\q \big[\cY^n_s -  V_n(s, \h X^{s_0, \bx^0, s})\big] dK^n_s = 0;\\
\displaystyle\tau^*_n := \inf\big\{s\ges s_0: \cY^n_s =   V_n(s, \h X^{s_0, \bx^0, s})\big\}.
\ea\right.
\ee
Note that
\beaa
& \cY^n_{s_0} \ges  V_n(s_0, \h X^{s_0, \bx^0, s_0}) =  V_n(s_0,  \bx^0) = [U_1-U_n](s_0, \bx^0) - {c\over 4} \ges {c\over 4} - [U-U_n](s_0, \bx^0),\\
& \cY^n_T = [U_1-U_n](t_0, T, \h X^{s_0, \bx^0, T}) \les [U-U_n] (t_0, T, \h X^{s_0, \bx^0, T}).
\eeaa
By \rf{fgn} and noting that $\bx\in \dbX$, one can easily show that
\bel{UUnest}
\lim_{n\to \infty} \dbE\Big[|OSC_{1\over n}( \h X^{s_0, \bx^0, \t})|^2+ |U-U_n|^2(t_0, \t,  \h X^{s_0, \bx^0, \t})\Big] =0,\q\mbox{for any}~\t\in \cT_{s_0}.
\ee
In particular, for any $\e>0$ small, this implies that,  for $n$ large enough,
\beaa
\cY^n_{s_0} \ges {c\over 8}>0\q\mbox{ and}\q  \dbE\big[ |\cY^n_T|^2{\bf 1}_{\{\cY^n_T\ges 0\}}\big] \les \e.
\eeaa
Denote $\th^n_s := L_0 {\rm sign}(\cZ^n_s)$.
Note that  $dK^n_s = 0$ for $s\in [s_0, \t^*_n]$, then $\cY^n_{s_0} = \dbE\big[M^{\th^n}_{\t^*_n} \cY^n_{\t^*_n}\big]$. Thus
\beaa
{c\over 8}&\les& \cY^n_{s_0} = \dbE\Big[M^{\th^n}_{\t^*_n} \cY^n_{\t^*_n}{\bf 1}_{\{\t^*_n<T\}} + M^{\th^n}_T \cY^n_T{\bf 1}_{\{\t^*_n=T\}}\Big]\\
&\les& \dbE\Big[M^{\th^n}_{\t^*_n} \cY^n_{\t^*_n}{\bf 1}_{\{\t^*_n<T\}} + M^{\th^n}_T \cY^n_T{\bf 1}_{\{\t^*_n=T\}}{\bf 1}_{\{\cY^n_T\ges 0\}}\Big]\\
&\les& C\Big[\sqrt{\dbP(\t^*_n<T)} + \sqrt{\e}\Big].
\eeaa
Then, for $\e>0$ small, we have
\bel{tau*n}
\dbP(\t^*_n<T) \ges {c^2\over C},\q\mbox{for all $n$ large enough}.
\ee
Moreover, by \rf{fgn} and  \rf{UUnest}, we have $\dbE[|\D_n|^2]\les \e^3$ for $n$ large enough, where
\bel{Deltan}
\D_n:=  \sup_{s\in [s_0, T], t\in [0, s]}\sup_{y,z} |f_n-f|(t,s, \h X^{s_0, \bx^0, s}, y, z) +|U-U_n|(t_0, \t^*_n,  \h X^{s_0, \bx^0, \t^*_n}),
\ee
which implies that
\beaa
\dbP(\D_n > \e)\les {1\over \e^2} \dbE[ |\D_n|^2]  \les \e.
\eeaa
Together with \rf{tau*n}, for $\e< {c^2\over C}$, we have $\dbP(\{\t^*_n < T\} \cap \{\D_n \les \e\}) >0$, for all $n$ large enough.
Therefore, there exists $\omega^*_n$ such that
\beaa
&\ds \t^*(\omega^*_n) < T, ~\D_n(\omega^*_n)\les \e,  ~\mbox{and}~ \psi_n(s_n^*, \bx^*_n)
=  V_n(s^*_n, \bx^*_n),\\
&\ds \mbox{where}\q s^*_n:=  \t^*_n(\omega^*_n),~ \bx^*_n:= \h X^{s_0, \bx^0, s^*_n}(\omega^*_n).
\eeaa
Now define
\beaa
\phi_n(s, \bx):= U_n(t_0, s, \bx) + {c\over 4(T-s_0)} (T-s) + \psi_n(s^*_n, \bx^*_n).
\eeaa
Similar to \rf{phi-A}--\rf{U1sub},  we have $\phi_n \in \underline \cA^{L_0}U_1(t_0, s_n^*, \bx_n^*)$ and then,
recalling that $T-s_0 \les \d = {1\over 8L}$,
\beaa
0&\les&  -2Lc+\Big[ \pa_s U_n(t_0,\cd) + {1\over 2} \langle \pa^2_{\bx\bx} U_n(t_0,\cd), (\si^{\cd}, \si^{\cd})\rangle
+ \langle \pa_\bx U_n(t_0,\cd), b^{\cd}\rangle \\
&&+ f\big(t_0, \cd, U_1(s^*_n,\cd), \langle \pa_\bx U_n(t_0,\cd), \si^{\cd}\rangle\big)\Big](s^*_n, \bx^*_n).
\eeaa
Since $U_n$ is a classical solution of the corresponding PPDE, this implies
\beaa
2Lc\les \Big[  f\big(t_0, \cd, U_1(s^*_n,\cd), \langle \pa_\bx U_n(t_0,\cd), \si^{\cd}\rangle\big)
- f_n\big(t_0, \cd, U_n(s^*_n,\cd), \langle \pa_\bx U_n(t_0,\cd), \si^{\cd}\rangle\big)\Big](s^*_n, \bx^*_n).
\eeaa
Then, by \rf{Deltan} and  recalling $\D_n(\omega^*_n)\les \e$, we have
\beaa
2Lc \les \Big[  f\big(t_0, \cd, U_1(s^*_n,\cd), \langle \pa_\bx U_n(t_0,\cd), \si^{\cd}\rangle\big)
- f\big(t_0, \cd, U(s^*_n,\cd), \langle \pa_\bx U_n(t_0,\cd), \si^{\cd}\rangle\big)\Big](s^*_n, \bx^*_n) +C\e.
\eeaa
Note further that \rf{c2} leads to $U_1(s^*_n, s^*_n, \bx^*_n) \les U(s^*_n, s^*_n, \bx^*_n)+c$,
and since $f$ is Lipschitz continuous and nondecreasing in $y$, we have $2Lc \les Lc +C\e$ and thus $Lc \les C\e$.
This is a desired contradiction since $\e$ can be arbitrarily small, thus $U_1(t,s,\bx) \les U(t,s,\bx)$ whenever $s\in [T-\d, T]$.

\ss

Now similar to the end of Theorem \ref{thm-partial}, we can show that  $U_1 \les U$ over the whole space.
\end{proof}

\begin{remark}\rm
In Theorem \ref{thm-comparison}, the assumptions imposed on $b,\si$ are somewhat strong.

(i) In \cite{EKTZ} and the subsequent works \cite{ETZ1,ETZ2},  general semi-martingale measures are used to define the corresponding set  of test functions $\underline \cA^L U$.  In this paper $X$ is not a semi-martingale and inside $U$ we need to use $\h X$, so in \rf{cA} we are using the exact process $\h X$. Consequently, we are not allowed to mollify $(b, \si)$, which will change the process $\h X$. Therefore, we assume $b, \si$ are smooth so that, together with mollified $(f_n, g_n)$, the corresponding PPDE has a classical solution $U_n$. It will be desirable to allow the $\h X$ in \rf{cA} to have more general distributions, in the spirit of \cite{ETZ1, ETZ2}.  Then it may become possible to mollify $(b, \si)$, and even to  allow $b, \si$ to depend on some controls. We shall leave this to future research.

(ii) If $X\equiv B^H$ is a fractional Brownian motion (with the Hurst parameter $H \neq {1\over 2}$),
namely $b\equiv 0$ and $\si(t,r,\bx)\equiv\si(t,r)$, following our arguments we may prove our results without Assumption \ref{assum-FSVIE} (iii) and
Assumption \ref{assum-classical} (iii).  Thus, in the setting that the randomness of $f,g$ comes from some fractional Brownian motions,
the  viscosity theory of the corresponding PPDEs still holds true.
\end{remark}

\section{Coupled Forward Backward SVIEs and Type-II BSVIEs}
\label{sect-extension}
\setcounter{equation}{0}
In this section, we investigate briefly two more general BSVIEs.

\subsection{Coupled FBSVIEs}

We now consider the following coupled FBSVIEs:
\begin{equation}\label{coupled-fbsvie}
\left\{\ba{lll}
\ds X_t=\bx_t+\int_0^t b(t,r,X_\cd,Y_r)dr+\int_0^t\si(t,r,X_\cd,Y_r)dW_r,\q t\in\dbT,\\
\ds Y_t=g(t,X_\cd) + \int_t^T f(t,r,X_\cd,Y_r,Z^t_r)dr-\int_t^T Z^t_rdW_r,\q t\in\dbT,
\ea\right.
\end{equation}
and the associated PPDE:
\bel{ppde-decouple}
\left\{\ba{lll}
\ns\ds \pa_s U(t, s, \bx) + {1\over 2} \langle \pa^2_{\bx\bx} U(t,s,\bx), (\h\si^{s,\bx}, \h\si^{s,\bx})\rangle + \langle \pa_\bx U(t,s,\bx), \h b^{s,\bx}\rangle \\
\ns\ds \q+ f\big(t, s, \bx, U(s,s, \bx), \langle \pa_\bx U(t,s,\bx), \h\si^{s,\bx}\rangle\big)=0,\q (t,s)\in\dbT_+^2,\,\bx\in\dbX,\\
\ns \ds U(t,T,\bx)=g(t,T,\bx),\q t\in\dbT,\,\bx\in\dbX,\ea\right.\ee
where, for $\f=b,\si$, $\h\f^{s,\bx}_r:=\f(r, s, \bx_\cd,U(s, s, \bx))$, $r\in[s,T]$.

When $T$ is small, Hamaguchi \cite{Hamaguchi-2020} proved the well-posedness of FBSVIE \rf{coupled-fbsvie} with the forward being a SDE. Following Ma--Protter--Yong \cite{Ma-Protter-Yong-1994}, in this subsection we prove the well-posedness of \rf{coupled-fbsvie} for arbitrary $T$, provided PPDE \rf{ppde-decouple} has a classical solution. The existence of such classical solution, as well as the well-posedness of \rf{coupled-fbsvie} in general, remains a challenging problem and we shall leave it for future research. For simplicity, in the following result we do not specify the precise technical conditions.

\begin{theorem}\label{thm-coupled-four} \sl
Assume $b,\si,f,g$ are sufficiently smooth with all the related derivatives being bounded. If PPDE \rf{ppde-decouple} has a classical solution $U$ with bounded $\pa_\bx U$, then the coupled FBSVIE \rf{coupled-fbsvie} admits a unique strong solution $(X, Y, Z)$ and the following representation holds:
\bel{YU-coulued}
\left.\ba{c}
\ds Y_t= U(t, t, \h X^t),\q Z^t_s=\lan\pa_\bx U(t, s, \h X^s),\,\h\si^{s,\h X^s}\ran,\q (t,s)\in\dbT_+^2,\\
\ns\ds \mbox{where}\q \h X^t := X\oplus_t \wt X_t,\q \wt X^s_t := \bx_s+\int_0^t b(s,r,X_\cd,Y_r)dr+\int_0^t\si(s,r,X_\cd,Y_r)dW_r.
\ea\right.\ee
\end{theorem}

\begin{proof}  We proceed in two steps. Fix an arbitrary $T$.

\ss
{\it Step 1.}  We first show the existence.  Let $\d>0$ be a small number which will be specified later. Introduce a mapping $\Phi$ on $\dbL^2_\dbF([0, \d]; \dbR^m)$ by $\Phi(\by) := Y^\by$, where,  for any $\by\in \dbL^2_\dbF([0, \d]; \dbR^m)$,
\bel{Xy}
\left.\ba{c}
\ds X^\by_t=\bx_t+\int_0^t b(t,r,X^\by_\cd,\by_r)dr+\int_0^t\si(t,r,X^\by_\cd, \by_r)dW_r,\q t\in [0, \d];\\
\ns\ds \wt X^{\by, s}_t =  \bx_s+\int_0^t b(s,r,X^\by_\cd,\by_r)dr+\int_0^t\si(s,r,X^\by_\cd, \by_r)dW_r,\q s\in [t, T];\\
\ns\ds \h X^{\by,t} := X^\by \oplus_t \wt X^\by_t,\q Y^\by_t := U(t,t, \h X^{\by,t}),\q t\in [0, \d].
\ea\right.
\ee
We emphasize that here we do not need to assume $T\les \d$. We shall show that $\Phi$ is a contraction mapping when $\d$ is small enough.

Indeed, let $\by, \by' \in  \dbL^2_\dbF([0, \d]; \dbR^m)$. Denote $\D \by := \by- \by'$, $\D X := X^\by - X^{\by'}$, and similarly for the other notations. First applying Proposition \ref{prop-FSVIE2} one can easily have
\bel{DXest1}
\dbE\Big[\sup_{0\les t\les \d} |\D X_t|^2\Big] \les C \dbE\Big[\int_0^\d |\D \by_r|^2 dr\Big].
\ee
Then by standard SDE estimates we have
\beaa
\dbE\big[ |\D \wt X^s_t|^2\big] \les C \dbE\Big[\int_0^\d |\D \by_r|^2 dr\Big],\q t\in [0, \d], s\in [t, T].
\eeaa
Moreover, since $\pa_t b, \pa_t \si$ satisfy the desired regularity, following the arguments in \cite{Ruan-Zhang-2020} we have
\beaa
\dbE\Big[\sup_{s\in [t,T]} |\D \wt X^s_t|^2\Big] \les C \dbE\Big[\int_0^\d |\D \by_r|^2 dr\Big],\q t\in [0, \d].
\eeaa
This, together with \rf{DXest1}, implies that
\beaa
\dbE\Big[\sup_{s\in [0,T]} |\D \h X^t_s|^2\Big] \les C \dbE\Big[\int_0^\d |\D \by_r|^2 dr\Big],\q t\in [0, \d].
\eeaa
Therefore, since $\pa_\bx U$ is bounded,
\beaa
\dbE[|\D Y_t|^2] \les C\dbE[\|\D \h X^t\|^2] \les C \dbE\Big[\int_0^\d |\D \by_r|^2 dr\Big],\q t\in [0, \d],
\eeaa
and thus
\beaa
\dbE\Big[\int_0^\d |\D Y_t|^2dt\Big] \les C\d \dbE\Big[\int_0^\d |\D \by_t|^2 dt\Big].
\eeaa
Choose $\d := {1\over 2C}$, we see that $\Phi$ is a contraction mapping. Consequently, $\Phi$ has a unique fixed point $\by^*$. Denote $X^*_t := X^{\by^*}_t, \wt X^{*,s}_t := \wt X^{\by^*,s}_t$, $t\in [0, \d]$, $s\in [t, T]$.

Next, we introduce another mapping $\Phi_2$ on $\dbL^2_\dbF([\d, 2\d]; \dbR^m)$ by $\Phi_2(\by) := Y^\by$, where, by abusing the notations, for any $\by\in \dbL^2_\dbF([\d,2\d]; \dbR^m)$,
\bel{Xy2}
\left.\ba{c}
\ds X^\by_t=\wt X^{*,t}_\d+\int_\d^t b(t,r, X^*\oplus_\d X^\by_\cd,\by_r)dr+\int_\d^t\si(t,r, X^*\oplus_t X^\by_\cd, \by_r)dW_r,\q t\in [\d,2\d];\\
\ns\ds \wt X^{\by, s}_t =  \wt X^{*,s}_\d+\int_\d^t b(s,r,X^*\oplus_\d X^\by_\cd,\by_r)dr+\int_\d^t\si(s,r,X^*\oplus_\d X^\by_\cd, \by_r)dW_r,\q s\in [t, T];\\
\ns\ds \h X^{\by,t} := X^*\oplus_\d X^\by \oplus_t \wt X^\by_t,\q Y^\by_t := U(t,t, \h X^{\by,t}),\q t\in [\d, 2\d].
\ea\right.
\ee
Following the same arguments we can show that $\Phi_2$ is also a contraction mapping, and thus we may extend the unique fixed point $\by^*$ to $[0, 2\d]$. Repeat the arguments we will obtain a fixed point $\by^*\in \dbL^2_\dbF([0, T]; \dbR^m)$ such that
\bel{Xy*}
\left.\ba{c}
\ds X^*_t=\bx_t+\int_0^t b(t,r,X^*_\cd,\by^*_r)dr+\int_0^t\si(t,r,X^*_\cd, \by^*_r)dW_r,\q t\in [0, T];\\
\ns\ds \wt X^{*,s}_t =  \bx_s+\int_0^t b(s,r,X^*_\cd,\by^*_r)dr+\int_0^t\si(s,r,X^*_\cd, \by^*_r)dW_r,\q s\in [t, T];\\
\ns\ds \h X^{*,t} := X^* \oplus_t \wt X^*_t,\q \by^*_t := U(t,t, \h X^{*,t}),\q t\in [0, T].
\ea\right.
\ee

Now applying the functional It\^o formula \rf{Ito-formula} on $U(t,s, \h X^{*,t})$ and utilizing the PPDE \rf{ppde-decouple}, one can easily see  that  $(X^*, \by^*)$ satisfy FBSVIE \rf{coupled-fbsvie} and the representation \rf{YU-coulued} holds true.

\ss

{\it Step 2.} We next show the uniqueness. Let $(X, Y, Z)$ be an arbitrary solution, and $\wt X, \wt Y, \h X$ be defined in an obvious way: $\h X^t := X\oplus_t \wt X_t$, and
\beaa
\left.\ba{lll}
\ds \wt X^s_t=\bx_s+\int_0^t b(s,r,X_\cd,Y_r)dr+\int_0^t\si(s,r,X_\cd,Y_r)dW_r,\\
\ds \wt Y^t_s=g(t,X_\cd) + \int_s^T f(t,r,X_\cd,Y_r, Z^t_r)dr-\int_s^T Z^t_rdW_r,
\ea\right. \q (t,s) \in \dbT^2_+.
\eeaa
Now denote $\f^{s, (\bx, y)}_r:= \f(r, s, \bx, y)$, $r\in [s, T]$, for $\f=b,\si$, and
\beaa
&\cY_t := U(t,t, \h X^t),\q \wt \cY^t_s := U(t,s, \h X^s),\q \cZ^t_s := \lan \pa_\bx U(t,s, \h X^s),  \si^{s,  (X, Y_s)}\ran;\\
&\D \cY := \cY - Y,\q \D \wt \cY := \wt \cY - \wt Y,\q \D \cZ:= \cZ - Z.
\eeaa
 Applying the functional It\^o formula \rf{Ito-formula}  and then utilizing the PPDE \rf{ppde-decouple}, we have
 \beaa
 d \wt \cY^t_s &=&  \cZ^t_s dW_s + \Big[\pa_s U + {1\over 2} \big\lan \pa^2_{\bx\bx} U, (\si^{s, X, Y_s}, \si^{s, X, Y_s})\big\ran
 +\big \lan \pa_{\bx} U, b^{s, X, Y_s}\big\ran \Big](t,s, \h X^s) ds\\
 &=& \cZ^t_s dW_s + \Big[ {1\over 2} \big\lan \pa^2_{\bx\bx} U,\, (\si^{s, X, Y_s}, \si^{s, X, Y_s})\big\ran
 +\big \lan \pa_{\bx} U, b^{s, X, Y_s}\big\ran \Big](t,s, \h X^s) ds\\
 &&-\Big[ {1\over 2} \big\langle \pa^2_{\bx\bx} U, (\h\si^{s, \h X^s}, \h\si^{s, \h X^s})\big\rangle + \big\langle \pa_\bx U, \h b^{s, \h X^s}\big\rangle
+ f\big(\cd, \cY_s, \langle \pa_\bx U, \h\si^{s, \h X^s}\rangle\big)\Big](t,s, \h X^s)ds.
 \eeaa
 Then
 \beaa
&& d\D \wt \cY^t_s =  \D\cZ^t_s dW_s +  {1\over 2} \Big[\big\lan \pa^2_{\bx\bx} U, (\si^{s, X, Y_s}, \si^{s, X, Y_s})\big\ran
- \big \langle \pa^2_{\bx\bx} U, (\h\si^{s, \h X^s}, \h\si^{s, \h X^s})\big\rangle \Big](t,s, \h X^s)ds \\
 &&\q + \big\lan \pa_{\bx} U, b^{s, X, Y_s} - \h b^{s, \h X^s} \big\ran(t,s, \h X^s) ds+\Big[  f\big(\cd, Y_s, Z^t_s\big)- f\big(\cd, \cY_s, \langle \pa_\bx U, \h\si^{s, \h X^s}\rangle\big)\Big](t,s, \h X^s)ds.
 \eeaa
 Noting that $\D \wt \cY^t_T=0$, by standard BSDE arguments, this implies
 \beaa
 \dbE\Big[|\D \wt \cY^t_s|^2 + \int_s^T |\D\cZ^t_r|^2dr\Big] &\les& C\dbE\Big[\int_s^T \Big(|\D \wt \cY^t_r|^2 + \| \si^{r, X, Y_r} -\h\si^{r, \h X^r}\|^2 +  \| b^{r, X, Y_r} -\h b^{r, \h X^r}\|^2 \\
 &&+ |\D \cY_r|^2 + |\D \wt \cY^t_r| | Z^t_r - \langle \pa_\bx U(t,r, \h X^r), \h\si^{r, \h X^r}\rangle|\Big) dr\Big].
 \eeaa
 Note that, for $\f = b, \si$,
 \beaa
&& \| \f^{r, X, Y_r} -\h\f^{r, \h X^r}\| = \sup_{l\in [r, T]} |\f(l, r, X, Y_r) - \f(l,r, X, U(r,r, \h X^r))| \les  C|\D \cY_r|;\\
 && | Z^t_r - \langle \pa_\bx U(t,r, \h X^r), \h\si^{r, \h X^r}\rangle| \les |\D \cZ^t_r| + \big|\langle \pa_\bx U(t,r, \h X^r), \si^{r, X, Y_r}-\h\si^{r, \h X^r}\rangle\big|\\
 &&\qq \les |\D \cZ^t_r| +  C\|\si^{r, X, Y_r}-\h\si^{r, \h X^r}\| \les |\D \cZ^t_r| +  C|\D \cY_r|.
 \eeaa
 Then
  \beaa
 \dbE\Big[|\D \wt \cY^t_s|^2 + \int_s^T |\D\cZ^t_r|^2dr\Big] \les \dbE\Big[\int_s^T \big[ C|\D \wt \cY^t_r|^2 + C|\D \cY_r|^2 + C |\D \wt \cY^t_r| | \cZ^t_r| \big] dr\Big],
 \eeaa
 which implies:
  \bel{DYest1}
  \dbE\big[|\D \wt \cY^t_s|^2\big]\les \dbE\Big[|\D \wt \cY^t_s|^2 +{1\over 2} \int_s^T |\D\cZ^t_r|^2dr\Big] \les C\dbE\Big[\int_s^T \big[ |\D \wt \cY^t_r|^2 + |\D \cY_r|^2 \big] dr\Big].
 \ee
 Now applying the Gr\"{o}nwall inequality we obtain
 \bel{DYest2}
  \dbE\big[|\D \wt \cY^t_s|^2\big]\les  C\dbE\Big[\int_s^T  |\D \cY_r|^2  dr\Big],\q s\in [t, T].
 \ee
 Set $s=t$ at above, we have
  \beaa
   \dbE\big[|\D \cY_t|^2\big] = \dbE\big[|\D \wt \cY^t_t|^2\big]\les  C\dbE\Big[\int_t^T  |\D \cY_r|^2  dr\Big],\q t\in [0, T].
 \eeaa
 Apply the Gr\"{o}nwall inequality again, we have $\D \cY=0$. Plug this into \rf{DYest2} and then to \rf{DYest1}, we see that $\D \wt \cY=\D\cZ = 0$.
In particular, this implies that $\by := Y$ is a fixed point in Step 1. By the uniqueness of the fixed point, we see that $Y$ is unique, which implies immediately that $X$ and $Z$ are also unique.
\end{proof}

\subsection{Type-II BSVIEs}
Let $X$ be the solution to FSVIE \rf{FSVIE}. In this subsection we consider the following  type-II BSVIE:
\bel{BSVIE-II}
\left.\ba{c}
\ds Y_t=g(t,X_\cd)+\int_t^T f(t,r,X_\cd,Y_r,Z^t_r,Z^r_t)dr -\int_t^TZ^t_rdW_r,\\
\ds Y_t=\dbE[Y_t]+\int_0^tZ^t_rdW_r,
\ea\right.\q t\in\dbT.
\ee
We note that here $f$ depends on both $Z^t_r$ and $Z^r_t$, where $Z^r_t$ for $t\les r$ is determined by the martingale representation of $Y_r$, as in the second line of \rf{BSVIE-II}. The $\dbF$-adapted solution to \rf{BSVIE-II} is called an {\it M-solution}, with {\it M} referring to martingale. By Yong \cite{Yong-2008}, under suitable conditions BSVIE \rf{BSVIE-II} admits a unique $M$-solution. Inspired by Wang--Yong \cite{Wang-Yong-2019}, we introduce the following PPDE:
\bel{PPDE-II}\left\{\begin{aligned}
&\pa_s V(t,s,\bx)+{1\over2}\lan\pa^2_{\bx\bx} V(t,s,\bx),
(\si^{s,\bx},\si^{s,\bx})\ran+\lan\pa_\bx V(t,s,\bx),b^{s,\bx}\ran=0,\q (t,s)\in\dbT^2_{-},\\
&\pa_s U(t,s, \bx',\bx)+{1\over2}\lan \pa^2_{\bx\bx}U(t,s,\bx',\bx),(\si^{s,\bx},\si^{s,\bx})\ran+\lan\pa_\bx U(t,s,\bx',\bx),b^{s,\bx}\ran\\
&\q+\1n f\big(t,s,\bx,U(s,s,\bx,\bx),\lan\pa_\bx U(t,s,\bx',\bx),\si^{s,\bx}\ran,\lan\pa_\bx V(s,t,\bx'),\si^{t,\bx'}\ran\big)=0,\q (t,s)\in\dbT^2_{+},\\
& V(t,t,\bx)=U(t,t,\bx,\bx),\q U(t,T,\bx',\bx)=g(t,\bx_\cd),\q t\in\dbT,~ \bx', \bx\in \dbX.
\end{aligned}\right.\ee
Note that here $V: \dbT_-^2 \times \dbX \to \dbR^m$ and $U: \dbT_+^2 \times \dbX\times \dbX \to \dbR^m$.

\begin{theorem}\label{repre-type-II} \sl
Assume $b,\si,f,g$ are sufficiently smooth with all the related derivatives being bounded, and let $X, \tilde X, \h X$ be determined by FSVIE \rf{FSVIE} in the obvious sense. Assume  PPDE \rf{PPDE-II} has a classical solution $(V,U)$ with bounded derivatives. Then the unique $M$-solution of  BSVIE \rf{BSVIE-II} satisfies: for any  $0\les r\les t\les s\les T$,
\bel{repre-type-II-main}
Y_t=U(t,t,\h X^t,\h X^t),\q Z^t_s=\lan\pa_\bx U(t,s,\h X^t, \h X^s),\si^{s,X}\ran,\q Z^t_r=\lan\pa_\bx V(t,r,\h X^r),\si^{r,X}\ran.
\ee
\end{theorem}
\begin{proof} Define $Y, Z$ as in \rf{repre-type-II-main} and $\wt Y^t_s := U(t,s,\h X^t, \h X^s)$, $\wt Y^t_r = V(t,r,\h X^r)$ for $0\les r\les t\les s\les T$.  We shall verify that they satisfy BSVIE \rf{BSVIE-II}.

First, fix $t$, and apply functional It\^o formula \rf{Ito-formula} on $V(t, r, \h X^r); r\in [0, t]$, we have
\beaa
d \wt Y^t_r =\Big[\pa_r V + {1\over2}\lan\pa^2_{\bx\bx} V, (\si^{r,X},\si^{r,X})\ran+\lan\pa_\bx V,b^{r,X}\ran\Big](t, r, \h X^r)ds  +  Z^t_r dW_r= Z^t_r dW_r,
\eeaa
where the second equality is due to  \rf{PPDE-II}. Since $\wt Y^t_t = Y_t$, this verifies the second line of \rf{BSVIE-II}.

Next, fix $(t, \h X^t)$, and apply functional It\^o formula \rf{Ito-formula} on $U(t, s, \h X^t, \h X^s);s\in [t, T]$, we have
\beaa
d \wt Y^t_s &=& \Big[\pa_s U+ {1\over2}\lan\pa^2_{\bx\bx} U, (\si^{s,X},\si^{s,X})\ran+\lan\pa_\bx U,b^{s,X}\ran\Big](t, s, \h X^t, \h X^s)ds  +  Z^t_s dW_s\\
&=& - f(t,s, Y_s, Z^t_s, Z^s_t) ds  +  Z^t_s dW_s,
\eeaa
where the second equality is also due to  \rf{PPDE-II}.  This verifies the first line of \rf{BSVIE-II} immediately.
\end{proof}

\section{Probabilistic representation of $\pa_\bx U(t, s,\bx)$}\label{sect-repre}
In this section, we shall investigate linear FBSVIEs more closely and then use it to obtain an explicit representation formula for $\pa_\bx U(t, s,\bx)$.

\subsection{A duality result for linear FSVIE}
For the ease of presentation, in the rest of the paper we restrict to one dimensional processes only.
However, all our results hold true in the multiple dimensional situation, and we provide a multiple dimensional setting in Remark \ref{rem-multiple} below.

Note that the dual space of $C([0, T])$ consists of signed measures on $[0, T]$ (see \cite{Yosida-1980}, p.119, for example).
That is, for a continuous linear mapping $\F: C([0, T]) \to \dbR$,
there exists a unique function $\check F$ on $[0, T]$ with finite variation such that $\lan \F, \eta \ran = \int_0^T \eta_t \check F(dt)$.
Then we may view the $D b, D\si$ in \rf{nablaX} as signed measures.
For this purpose, in this subsection we consider the following FSVIE:
\bel{measure-FSVIE}
\cX_t = \eta_t + \int_0^t  \int_0^s \cX_r~ \check b(t, s, dr) ds +  \int_0^t  \int_0^s \cX_r ~\check \si(t, s, dr) dW_s.
\ee
Here $\check b, \check \si: (t,s,r,\omega)\in \dbT_-^3 \times \Omega \to \dbR$  are progressively measurable,
the adaptedness of $\omega$ is with respect to the second time variable $s$,
and the dependence on the third time variable $r$ is right continuous with finite variation. We are interested in the term
\bel{measure-Y0}
 \dbE\Big[ \int_0^T \cX_r ~\check g(dr)\Big],
\ee
where $\check g: \dbT\times \Omega\to \dbR$ is $\cF_T$ measurable in $\omega$ and right continuous and finite variated in $r$.
Our goal is to find a finite variated function $r\in \dbT \mapsto \wt \cY(r, 0)$  such that the following duality principle holds:
\bel{duality}
 \dbE\Big[ \int_0^T \cX_r \check g(dr) \Big] = \int_0^T \eta_r \wt \cY(dr, 0).
 \ee
 This will give us an explicit representation for the linear mapping $\check F$:
 \bel{rep}
 \Phi =\check F(\cd)= \wt\cY(\cd, 0).
 \ee

We shall approach the problem dynamically. Define, for $0\les t\les r \les T$,
\bea
\label{wtcX}
\wt \cX^r_t := \eta_r +  \int_0^r  \int_0^{s\wedge t} \cX_l~ \check b(r, s, dl) ds +  \int_0^r  \int_0^{s\wedge t} \cX_l ~\check \si(r, s, dl) dW_s.
\eea
 Note that $\eta_r = \tilde X^r_0$. In light of \rf{duality}, we want to find $\wt\cY$ such that
 \bel{duality2}
 \dbE_t\Big[ \int_t^T \cX_r \check g(dr) \Big] =  \int_t^T \wt \cX^r_t \wt \cY(dr, t), \q t\in [0, T].
 \ee
 For this purpose, we introduce the following type-II BSVIE:
 \bel{II-BSVIE-dual}
 \left.\ba{c}
 \ds \cY_t = \check g(t)  - \int_{t\les l\les s\les r\les T}  \Big[ \check b(r,s, dl)  \wt \cY(dr, s) + \check \si(r,s, dl) \cZ(dr, s)\Big]ds - \int_t^T \cZ(t, s) dW_s;\\
 \ds \wt \cY(t, s) = \cY_t - \int_s^t \cZ(t, r) dW_r,\q 0\les s\les t.
 \ea\right.
 \ee
 We emphasize that, for fixed $t$,
\ss

 $\bullet$ The mapping $s\in [0, t] \to \wt \cY(t,s)$ is an $\dbF$-martingale;

\ss
 $\bullet$ The mappings $s\in [t, T] \to \big(\wt \cY(s,t), \cZ(s, t)\big)$ are $\cF_t$-measurable and  finite variated.

\ss
 \no The second requirement above, of course, will add difficulty for the existence of solutions, which we shall leave for future research.

 \begin{theorem}
 \label{thm-Fduality}
 Let $\cX$ and $\cY, \wt \cY, \cZ$ be the solution to \rf{measure-FSVIE} and \rf{II-BSVIE-dual}, respectively.
 Then \rf{duality2} holds, and in particular \rf{duality} holds.
 \end{theorem}
 \begin{proof}
 We shall only prove \rf{duality}, the arguments for \rf{duality2} are similar.

Since $\eta$ is continuous, by taking time partitions we have
  \beaa
   \int_0^T \eta_t \wt \cY(dt,0) =\lim_{N\to\infty} \sum_{i=0}^{N-1}  \big[\wt \cY(t_{i+1},0) -  \wt \cY(t_i,0)\big]\eta_{t_i}.
   \eeaa
Now fix a time partition with large $N$,  by the second line of \rf{II-BSVIE-dual}, we see that
 \beaa
&& \int_0^T \eta_t \wt \cY(dt,0) \approx \sum_{i=0}^{N-1} \big[\wt \cY(t_{i+1},0) -  \wt \cY(t_i,0)\big] \eta_{t_i}
= \sum_{i=0}^{N-1}\dbE\Big[  \big[\cY_{ t_{i+1}} -   \cY_{t_i}\big] \eta_{t_i}\Big]\\
&&\q= \sum_{i=0}^{N-1} \dbE\Big[\big[\cY_{ t_{i+1}} -   \cY_{t_i}\big]  \big[\cX_{t_i} - \int_0^{t_i}  \int_0^s \cX_r~ \check b(t_i, s, dr) ds -  \int_0^{t_i}  \int_0^s \cX_r ~\check \si(t_i, s, dr) dW_s\big] \Big]\\
&&\q= \sum_{i=0}^{N-1} \dbE\Big[ \big[\cY_{ t_{i+1}} -   \cY_{t_i}\big]\cX_{t_i} -\big[\cY( t_{i+1},t_i) -   \cY(t_i,t_i)\big]\\
&&\qq\qq\q\times \big[ \int_0^{t_i}  \int_0^s \cX_r~ \check b(t_i, s, dr) ds + \int_0^{t_i}  \int_0^s \cX_r ~\check \si(t_i, s, dr) dW_s\big]\Big]\\
&&\q
= \sum_{i=0}^{N-1} \dbE\Big[ \big[\cY_{ t_{i+1}} -   \cY_{t_i}\big]\cX_{t_i} -  \int_0^{t_i}\big[\wt \cY( t_{i+1}, s) -  \wt \cY( t_i, s)\big] \big[ \int_0^s \cX_r~ \check b(t_i, s, dr)\big]  ds\\
&&\q\qq\qq - \int_0^{t_i} \big[\cZ(t_{i+1}, s) - \cZ(t_i,s)\big]\big[ \int_0^s \cX_r~ \check \si(t_i, s, dr)\big]  ds \Big]\\
&&\q\approx  \sum_{k=0}^{N-1} \dbE\Big[ \big[\cY_{ t_{k+1}} -   \cY_{t_k}\big]\cX_{t_k}\Big] - {T\over n}\sum_{0
\les k<j<i\les N-1} \Big[ \big[\wt \cY( t_{i+1}, t_j) -  \wt \cY( t_i, t_j)\big] \\
&&\q\q\times\big[\check b(t_i, t_j, t_{k+1}) - \check b(t_i, t_j, t_k)\big]+ \big[\cZ( t_{i+1}, t_j) -  \cZ( t_i, t_j)\big] \big[\check \si(t_i, t_j, t_{k+1}) - \check \si(t_i, t_j, t_k)\big] \Big]\cX_{t_k}.
\eeaa
Here and in the sequel, we are using $\approx$ to denote a difference of $o(1)$ term when $N\to\infty$. Then
\bel{app1}
\left.\ba{lll}
\ds\int_0^T \eta_t \wt \cY(dt,0) - \dbE\Big[\int_0^T \cX_t \check g(dt)\Big]\\
\ds \q\approx  \int_0^T \eta_t \wt \cY(dt,0) - \sum_{k=0}^{N-1} \dbE\Big[ \big[\check g(t_{k+1}) - \check g(t_k)\big]\cX_{t_k} \Big]
\approx   \sum_{k=0}^{N-1} \dbE\big[  \cX_{t_k}I_k \big],
\ea\right.
\ee
where, for each $k$,
\beaa
 I_k &:=&  \int_{t_k\les l\les s\les r\les T}  \Big[ \check b(r,s, dl)  \wt \cY(dr, s) + \check \si(r,s, dl) \cZ(dr, s)\Big]ds \\
&&-  \int_{t_{k+1}\les l\les s\les r\les T}  \Big[ \check b(r,s, dl)  \wt \cY(dr, s) + \check \si(r,s, dl) \cZ(dr, s)\Big]ds\\
&&- {T\over n}\sum_{k<j<i\les N-1} \Big[ \big[\check b(t_i, t_j, t_{k+1}) - \check b(t_i, t_j, t_k)\big] \big[\wt \cY( t_{i+1}, t_j) -  \wt \cY( t_i, t_j)\big] \\
&&\qq + \big[\check \si(t_i, t_j, t_{k+1}) - \check \si(t_i, t_j, t_k)\big] \big[\cZ( t_{i+1}, t_j) -  \cZ( t_i, t_j)\big] \Big].
\eeaa
One may easily check that
\beaa
I_k
&\approx&  \int_{t_k\les s\les r\les T}  \Big[ \big[\check b(r,s, s) - \check b(r,s, t_k)\big] \wt \cY(dr, s) + \big[\check \si(r,s, s) - \check \si(r,s, t_k)\big] \cZ(dr, s)\Big]ds \\
&&-  \int_{t_{k+1}\les s\les r\les T}  \Big[ \big[\check b(r,s, s) - \check b(r,s, t_{k+1})\big]  \wt \cY(dr, s)
+ \big[\check \si(r,s, s) - \check \si(r,s, t_{k+1})\big] \cZ(dr, s)\Big]ds\\
&&-\int_{t_{k+1}\les s\les r\les T} \Big[ \big[\check b(r, s, t_{k+1}) - \check b(r, s, t_k)\big] \wt \cY( dr, s)
+ \big[\check \si(r, s, t_{k+1}) - \check \si(r, s, t_k)\big] \cZ( dr, s) \Big] ds \\
&=& \int_{t_k\les s\les r\les T}  \Big[ \big[\check b(r,s, s) - \check b(r,s, t_k)\big] \wt \cY(dr, s)
+ \big[\check \si(r,s, s) - \check \si(r,s, t_k)\big] \cZ(dr, s)\Big]ds \\
&&-  \int_{t_{k+1}\les s\les r\les T}  \Big[ \big[\check b(r,s, s) - \check b(r,s, t_{k})\big]  \wt \cY(dr, s)
 + \big[\check \si(r,s, s) - \check \si(r,s, t_{k})\big] \cZ(dr, s)\Big]ds\\
&=& \int_{t_k}^{t_{k+1}} \int_s^T  \Big[ \big[\check b(r,s, s) - \check b(r,s, t_k)\big] \wt \cY(dr, s)
+\big [\check \si(r,s, s) - \check \si(r,s, t_k)\big] \cZ(dr, s)\Big]ds.
\eeaa
Substituting the above into \rf{app1} implies that
\bel{appro1}
\begin{aligned}
& \int_0^T \eta_t \wt \cY(dt,0) - \dbE\Big[\int_0^T \cX_t \check g(dt)\Big] \approx\sum_{k=0}^{N-1} \dbE\Big[\int_{t_k}^{t_{k+1}} \int_s^T  \Big[ \big[\check b(r,s, s) - \check b(r,s, t_k) \big]\wt \cY(dr, s)\\
&\q \qq +\big [\check \si(r,s, s) - \check \si(r,s, t_k) \big]\cZ(dr, s)\Big]ds\cX_{t_k}\Big].
\end{aligned}\ee
Using the fact that the finite variated function is a.e continuous, we get
  \beaa
 \lim_{t_k\uparrow s}\big [\check b(r,s, s) - \check b(r,s, t_k)\big]
  =  \lim_{t_k\uparrow s} \big[ \check \si(r,s, s) - \check \si(r,s, t_k)\big] = 0, \q\mbox{for a.e. $s$}.
  \eeaa
Then from \rf{appro1}  we see that  \rf{duality} holds true by letting $N\to\i$.
 \end{proof}

\begin{remark}
\label{rem-measure}
\rm  In the state dependent case, the measures are degenerate:
\beaa
\check b(t,s, dr) = b(t,s) \d_s(r), \q \check \si(t,s,dr) = \si(t,s) \d_{s}(r),\q \check g(dr) =  g(r) dr,
\eeaa
with the Dirac measure $\d$ being defined by
$$
\d_s\bx=\bx_s,\q s\in[0,T].~\bx\in C([0,T]).
$$
Then \rf{measure-FSVIE} and  \rf{II-BSVIE-dual} become
\bel{II-state}
\left.\ba{l}
\ds \cX_t = \eta_t + \int_0^t  b(t,s) \cX_s ds +  \int_0^t \si(t, s) \cX_s dW_s;\\
 \ds \cY_t = \check g(t)  - \int_t^T \int_s^T  \Big[ b(r,s)  \wt \cY(dr, s) +\si(r,s) \cZ(dr, s)\Big]ds - \int_t^T \cZ(t, s) dW_s;\\
 \ds  \wt \cY(t, s) = \cY_t - \int_s^t \cZ(t, r) dW_r,\q 0\les s\les t.
\ea\right.
\ee
Let $(Y, Z)$ denote the solution to the following type-II BSVIE:
\bel{II-Yong}
\left.\ba{l}
 \ds Y_t = g(t)  + \int_t^T  [b(s,t)  Y_s +\si(s,t) Z(s, t)]ds - \int_t^T Z(t, s) dW_s;\\
 \ds Y_t = \dbE[Y_t] +\int_0^t Z(t,s) dW_s.
\ea\right.
\ee
We can easily check that
\beaa
\cY_t = \cY_0 + \int_0^t Y_r dr, \q \wt \cY(t,s) = \cY_s + \int_s^t \dbE_s[Y_r] dr,\q \cZ(t,s) = \int_s^t Z(r,s) dr,\q 0\les s\les t\les T
\eeaa
satisfy the BSVIE  in \rf{II-state}. Then \rf{duality} becomes
\beaa
\dbE\Big[ \int_0^Tg(t) \cX_t dt\Big] =  \dbE\Big[ \int_0^T \cX_t\check g(dt)\Big]= \int_0^T \eta_t \wt \cY(dt, 0)
 =  \int_0^T \eta_t   \dbE[Y_t] dt = \dbE\Big[\int_0^T \eta_t Y_tdt\Big].
\eeaa
This is exactly the duality in Yong \cite{Yong-2006, Yong-2008}.
So our result here is a generalization of these works.
\end{remark}
\begin{remark}
\label{rem-measure2}
\rm  Our result also generalizes the duality between delayed SDEs and anticipated BSDEs in Peng--Yang \cite{Peng-Yang}. Let $(X,Y,Z)$ denote the solution to the following equations:
\bel{SDDE}\left\{\ba{ll}
\ds dX^\xi_s=(\mu_sX^\xi_s+\bar\mu_{s-\th}X^\xi_{s-\th})ds+(\si_s X^\xi_s+\si_{s-\th} X^\xi_{s-\th}) dW_s,\q s\in[t,T+\th],\\
\ns\ds X_t=\xi,\q X_s=0,\q s\in[t-\th,t];\\
\ns\ds-dY_s=\big(\mu_s Y_s+\bar\mu_{s}\dbE_s[Y_{s+\th}]+\si_s Z_s+\si_{s}\dbE_s[Z_{s+\th}]+l_s\big)ds-Z_s dW_s,\q s\in[t,T],\\
\ns\ds Y_s=Q_s,~~ Z_s=P_s,\q s\in[T,T+\th],
\ea\right.\ee
where $\th>0$ is a fixed delay time.
From \cite[Theorem 2.1]{Peng-Yang} we get the duality
\bel{ab-duality}
\lan Y_{t},\xi\ran=\dbE_t\Big[X^\xi_T Q_T+\int_t^T X^\xi_s l_s ds+\int_T^{T+\th}(Q_s \bar \mu_{s-\th}+P_s\bar\sigma_{s-\th})X^\xi_{s-\th}ds\Big]:= \dbE_t\[\int_t^T X^\xi_s\check g(ds)\],
\ee
which shows that $Y_t$ is an explicit representation of the linear functional $\xi\mapsto\dbE_t\big[\int_t^T X^\xi_s\check g(ds)\big]$.
Since FSVIE  \rf{measure-FSVIE} is more general than the delayed SDE in \rf{SDDE},
we can also use Theorem \ref{thm-Fduality} to give such an  explicit representation for $\xi\mapsto\dbE_t\big[\int_t^{T}X^\xi_s\check g(ds)\big]$.
Indeed,  take
$$
\begin{aligned}
&\check b(s, dr) = \mu_s \d_s(r)+\bar\mu_{s-\th}\d_{s-\th}(r), \q \check \si(s,dr) = \si(s) \d_{s}(r)+\bar\si_{s-\th}\d_{s-\th}(r),\q s\in[t+\th,T],\\
&\check b(s, dr) = \mu_s \d_s(r), \q \check \si(s,dr) = \si(s) \d_{s}(r),\q s\in[t,t+\th],\\
& \check g(dr) =  l(r)dr+\dbE_r[Q_{r+\th} \bar \mu_{r}+P_{r+\th}\bar\sigma_{r}]1_{[T-\th,T]}(r)dr+Q(T)\d_T(r).
\end{aligned}
$$
Note that $\check b(\t,s, dr), \check\si(\t,s, dr)$ are independent of $\t$, the corresponding BSVIE \rf{II-BSVIE-dual} reads
$$
 \left.\ba{c}
 \ds \cY_\t = \check g(\t)  - \int_{\t}^T \mu_s[\wt\cY(T,s)-\wt\cY(s,s) ]ds - \int_{\t+\th}^T \bar\mu_{s-\th}[\wt\cY(T,s)-\wt\cY(s,s) ]ds\\
 \ds \qq\qq- \int_{\t}^T \si_s[\cZ(T,s)-\cZ_s ]ds - \int_{\t+\th}^T \bar\si_{s-\th}[\cZ(T,s)-\cZ_s ]ds-\int_\t^T \cZ_sdW_s;\\
 \ds \wt \cY(\t, s) = \cY_\t - \int_s^\t \cZ(\t,r) dW_r,\q 0\les s\les \t.
 \ea\right.
$$
Then it is easy to check
\begin{align*}
&Y_t=\cY(T,t)-\cY_t,\q Z_t=\cZ(T,t)-\cZ_t, \\
&\lan \cY(T,t)-\cY_t,\,\xi\ran=\lan Y_{t},\,\xi\ran= \dbE_t\[\int_t^T X^\xi_s\check g(ds)\].
\end{align*}
Thus Theorem \ref{thm-Fduality}  covers the duality in \cite{Peng-Yang}.
\end{remark}

\begin{remark}
\label{rem-multiple}\rm
The duality \rf{duality2} still holds true in the multidimensional case, where the FSVIE \rf{measure-FSVIE} and type-II BSVIE \rf{II-BSVIE-dual} become
$$
 \left.\ba{l}
 \ds\cX_t = \eta_t + \int_0^t  \int_0^s \cX_r~ \check b(t, s, dr) ds +  \sum_{ j=1}^d\int_0^t  \int_0^s \cX_r ~\check \si^{ j}(t, s, dr) dW^{ j}_s;\\
  \ds\cY_t = \check g(t)  - \int_{t\les l\les s\les r\les T}  \Big[ \check b(r,s, dl)  \wt \cY(dr, s) +  \sum_{ j=1}^d\check \si^{ j}(r,s, dl) \cZ^{ j}(dr, s)\Big]ds
  -\sum_{ j=1}^d\int_t^T \cZ^{ j}(t, s) dW_s^{ j};\\
  \ds\wt \cY(t, s) = \cY_t -  \sum_{ j=1}^d\int_s^t \cZ^{ j}(t, r) dW^{ j}_r,\q 0\les s\les t,
 \ea\right.
$$
with $\check b, \check \si^{ j}: (t,s,r,\omega)\in \dbT_-^3 \times \Omega \to \dbR^{n\times n}$ and $\check g: \dbT\times \Omega\to \dbR^{m\times n}$ being proper maps.
\end{remark}

\subsection{An explicit solution for linear BSVIEs}
In this subsection we investigate the following linear BSVIE:
\bel{linearBSVIE}
 \cY_t = \xi_t + \int_t^T \big[\a(t, r) \cY_r + \beta(t, r) \cZ^t_r \big] dr - \int_t^T \cZ^t_r dW_r,
\ee
where $\xi: \dbT \times \Omega \to \dbR$,  $\a, \beta: \dbT_+^2 \times \Omega \to \dbR$ are progressively measurable (omitting the variable $\omega$).

\begin{proposition}\label{variation-formula} \sl
Assume $\a, \beta$ are bounded and $\sup_{t\in \dbT} \dbE[ |\xi_t|^2]<\infty$. Then
\bel{linear-representation}
 \cY_t =  \dbE_t\Big[ M^t_T \xi_t + \int_t^T \G(t, r) M^r_T \xi_r dr\Big],
 \ee
where $M$ is the solution to the following SDE:
\bel{linear-representation1}
 d M^t_r = M^t_r \beta(t, r)dW_r,~~(t,r)\in\dbT^2_+;\q M_t^t=I_m,
\ee
and
\bel{linear-representation2}
\G(t,s) := \sum_{n=1}^\infty K_n(t, s), ~~ K_1(t,s):= M^t_s \a(t,s), ~~ K_{n+1}(t, s) := \int_t^s K_1(t, r) K_n(r, s) dr.
\ee
\end{proposition}
\begin{proof} First, by Proposition \ref{prop-BSVIE2} we see that \rf{linearBSVIE} is wellposed.  Next, since $\a, \beta$ are bounded, it is clear that $\dbE_t[|K_1(t,s)|^2] \les C_0<\infty$.  Note that
\beaa
\dbE_t[|K_{n+1}(t, s)|^2] \les (s-t) \int_t^s \dbE_t\Big[|K_1(t, r)|^2 \dbE_r[ |K_n(r, s)|^2]\Big] dr.
\eeaa
Then by induction one can easily show that
\beaa
\dbE_t[|K_{n+1}(t, s)|^2] \les {C_0^{n+1}(s-t)^{2n}\over (2n-1)!!},\q\mbox{and thus}\q \dbE_t[|\G(t,s)|^2] \les C < \infty.
\eeaa
We now let $(\wt\cY,\wt\cZ)$ satisfy the following BSDE:
\bel{linearBSVIE1}
\wt\cY^t_s = \xi_t + \int_s^T \big[\a(t, r) \cY_r + \beta(t, r) \wt\cZ^{ t}_r \big] dr - \int_s^T  \wt\cZ^{ t}_r dW_r.
\ee
Then
$$
\cY_t= \wt\cY^t_t,\q \cZ^t_s=\wt\cZ^t_s.
$$
Apply It\^o formula to the mapping $s\mapsto M^t_s\wt\cY^t_s$ on $[t,T]$, we get
\bel{linearBSVIE2}
\cY_t =\wt\cY_t^t = \dbE_t \Big[ M^t_T \xi_t + \int_t^T M^t_s \a(t, s)\cY_s ds\Big].
\ee
Moreover, note that
$$\G(t,s) = K_1(t, s) + \int_t^s K_1(t, r) \G(r, s) dr.$$
Then
\beaa
\cY_t  &=&  \dbE_t \Big[ M^t_T \xi_t + \int_t^T K_1(t, s)\dbE_s\big[ M^s_T \xi_s + \int_s^T \G(s, r) M^r_T \xi_r dr\big] ds\Big]\\
  &=&  \dbE_t \Big[ M^t_T \xi_t + \int_t^T K_1(t, s) M^s_T \xi_s ds+ \int_t^TK_1(t, s) \int_s^T \G(s, r) M^r_T \xi_r dr ds\Big]\\
  &=&  \dbE_t \Big[ M^t_T \xi_t + \int_t^T K_1(t, s) M^s_T \xi_s ds+ \int_t^T \int_t^r K_1(t, s) \G(s, r) ds  M^r_T \xi_r dr \Big]\\
  &=&  \dbE_t \Big[ M^t_T \xi_t + \int_t^T \big[K_1(t, s)+ \int_t^s K_1(t, r) \G(r, s) dr\big] M^s_T \xi_s ds \Big]\\
  &=&  \dbE_t \Big[ M^t_T \xi_t + \int_t^T \G(t, s)  M^s_T \xi_s ds \Big] = \cY_t.
\eeaa
This implies that $(\cY, \cZ)$ satisfy \rf{linearBSVIE}. The result then follows from the uniqueness of \rf{linearBSVIE}.
\end{proof}

\begin{remark}
\label{rem-linear}\rm
The representation \rf{linear-representation} of $\cY$ is exactly the so-called {\it variation of constants formula} for linear BSVIEs.
A similar result was first obtained by Hu--{\O}ksendal \cite{Hu 2018} for the linear BSVIEs driven by a Brownian motion and a compensated Poisson random measure. However, in  \cite{Hu 2018} the coefficients $\a, \beta$  are assumed to be  deterministic functions and $\beta(t,r)\equiv \beta(r)$ is required to be independent of $t$. Thus our result is a generalized version of  \cite[Theorem 3.1]{Hu 2018}.
\end{remark}

\subsection{Representation of $\pa_\bx U$}
\label{sect-paxU}
In this subsection we assume Assumption \ref{assum-classical} holds true, and let $U$ be the classical solution to PPDE \rf{UPPDE}, corresponding to the decoupled FBSVIE \rf{FSVIE}--\rf{BSVIE}.  We shall use type-II BSVIE to provide an explicit representation formula for $\pa_\bx U(t,s,\bx)$, which is determined by \rf{paxU}--\rf{nablaX}.

We first apply Proposition \ref{variation-formula} to the middle equation of \rf{nablaX} with
\bel{ab}
\left.\ba{l}
\ds\xi_l:= \big\lan Dg(l ,X^{s,\bx}),\nabla_\eta X^{s,\bx}\big\ran
+\int_{l\vee s} ^T\big\lan Df\big(l ,r,X^{s,\bx},Y^{s,\bx}_r, Z^{l ,s,\bx}_r\big),\nabla_\eta X^{s,\bx}\big\ran dr;\\
\ds\a(l, r) := \pa_yf\big(l ,r,X^{s,\bx},Y^{s,\bx}_r, Z^{l ,s,\bx}_r\big),\q \beta(l, r) := \pa_zf\big(l ,r,X^{s,\bx},Y^{s,\bx}_r, Z^{l ,s,\bx}_r\big),\q l \vee s \les r;\\
\ns\ds \a(t, r) := 0, \q \beta(t, r) := 0,\q t\les r<s.
\ea\right.
\ee
Define $M^l_r$, $K_1(l, r)$,  and $\G(l, r)$ by \rf{linear-representation1}--\rf{linear-representation2}, then
\beaa
\nabla_\eta Y^{s,\bx}_l = \dbE_l\Big[M^l_T \xi_l + \int_l^T \G(l, r) M^r_T \xi_r dr\Big],\q l\in[s,T].
\eeaa
Note that $M^t_s = 1$ and $K_1(t, r)=0$ for $r\in [t, s]$, thanks to the third line of \rf{ab}. Then, by \rf{paxU} and the last equation of \rf{nablaX} we have
\beaa
\pa_\bx U(t,s,\bx) &=& \nabla_\eta\wt Y^{t,s,\bx}_s = \dbE\Big[M^t_T \xi_t + \int_s^T  M^t_l \a(t, l) \nabla_\eta Y^{s,\bx}_ldl\Big]\nonumber\\
&=&\dbE\Big[M^{t}_T \xi_t + \int_s^T  K_1(t, l)\big[M^l_T \xi_l + \int_l^T \G(l, r) M^r_T \xi_r dr\big]dl\Big]\nonumber\\
&=&\dbE\Big[M^{t}_T \xi_t + \int_s^T  \big[K_1(t,l) + \int_s^r K_1(t, l) \G(l, r) dl \big]M^r_T \xi_r dr\Big]\nonumber\\
&=&\dbE\Big[M^{t}_T \xi_t + \int_s^T  \G(t, r)M^r_T \xi_r dr\Big].
\eeaa
Plug the first line of \rf{ab} into this, we obtain
\bel{paxU2}
\left.\ba{c}
\ds\pa_\bx U(t,s,\bx) = \dbE\big[ \lan G^{s,\bx}(t),\nabla_\eta X^{s,\bx}\ran \big], \\
\ns\ds \mbox{where}\q G^{s,\bx}(t) := M^t_T D g(t, X^{s,\bx}) + \int_s^T \G(t, l) M^l_T D g(l, X^{s,\bx}) dl + \int_s^T H^{s,\bx}(t, r)dr;\\
\ns\ds H^{s,\bx}(t,r) := M^t_r Df\big(t ,r,X^{s,\bx},Y^{s,\bx}_r, Z^{t ,s,\bx}_r\big) + \int_s^r \G(t, l) M^l_r Df\big(l ,r,X^{s,\bx},Y^{s,\bx}_r, Z^{l ,s,\bx}_r\big) dl.
\ea\right.
\ee

Next, recall the first equation of \rf{nablaX}. We set
\bel{checkb}\begin{aligned}
&\check \f(t,s,\bx; t', s', dr') := D\f (t' ,s',X^{s,\bx}) (dr'),\q \hbox{for~} \f = b, \si;\\
& \check g (t, s, \bx; dt') := G^{s,\bx}(t) (dt').
\end{aligned}\ee
We now introduce the type-II BSVIE on $[s, T]$:
\bel{II-BSVIE-dual2}
\begin{aligned}
  &\cY_{t'} = \check g(t') - \int_{t'\les l'\les s'\les r'\les T}  \Big[ \check b(t,s,\bx; r',s', dl')  \wt \cY(dr', s')  \\
&\qq\q + \check \si(t,s,\bx; r',s', dl') \cZ(dr', s')\Big]ds'- \int_{t'}^T \cZ(t', s') dW_{s'};\\
 & \wt \cY(t', s') = \cY_{t'} - \int_{s'}^{t'} \cZ(t', r') dW_{r'},\q 0\les s'\les t'.
 \end{aligned}
 \ee
By Theorem \ref{thm-Fduality}, we obtain the following explicit representation formula for $\pa_\bx U(t,s,\bx)$.
\begin{theorem}\label{thm-repre-Ux} \sl
For any fixed $(t,s,\bx)\in\dbT^2_+\times \dbX$, let $\wt\cY$ be determined by \rf{II-BSVIE-dual2}.
Then  the path derivative of the solution $U$ to PPDE \rf{UPPDE} can be represented explicitly as follows:
\bel{paxU3}
\pa_\bx U(t,s,\bx) = \wt \cY(\cd, s).
\ee
\end{theorem}

%
%


\end{document}